\documentclass[a4paper, reqno]{amsart}
\usepackage[foot]{amsaddr}
\usepackage{amsthm,amssymb}

\usepackage{tikz}
\usetikzlibrary{shapes,arrows,positioning}

\usepackage[utf8]{inputenc}
\usepackage{amsmath,amssymb,amsthm,commath,esint,tikz-cd,tikz, mathtools, physics}
\usepackage{natbib}
\usepackage[mathscr]{euscript}
\usepackage{amsfonts}
\usepackage{subcaption}
\usepackage{hyperref}
\usepackage{graphicx}
\usepackage{float}
\usepackage{fancyhdr}
\usepackage{bm}
\usepackage{bbm}

\numberwithin{equation}{section}

\title[PINNs for Navier-Stokes equations]{Error estimates for physics-informed neural networks approximating the Navier-Stokes equations}

\author{Tim~De Ryck}
\email{tim.deryck@sam.math.ethz.ch}

\author{Ameya D.~Jagtap}
\email{ameya\_jagtap@brown.edu}

\author{Siddhartha~Mishra}
\email{siddhartha.mishra@sam.math.ethz.ch}

\address[T. De Ryck]{Seminar for Applied Mathematics, D-MATH, ETH Z\"urich, Rämistrasse 101, 8092 Zürich, Switzerland. }
\address[S. Mishra]{Seminar for Applied Mathematics, D- MATH, and ETH AI Center, ETH Z\"urich, Rämistrasse 101, 8092 Zürich, Switzerland. }

\address[A.D. Jagtap]{Division of Applied Mathematics, Brown University, 182 George street, Providence, RI 02912, USA.}

\renewcommand{\S}{\mathcal{S}} 
\newcommand{\Et}{\mathcal{E}_T} 
\newcommand{\Eg}{\mathcal{E}_G} 
\newcommand{\Etot}{\mathcal{E}}

\newcommand{\hu}{\hat{u}}
\newcommand{\N}{\mathbb{N}}
\newcommand{\R}{\mathbb{R}}

\newcommand{\bigO}{\mathcal{O}}

\newcommand{\fhat}{\hat{f}}

\renewcommand{\div}[1]{\mathrm{div}\left( #1 \right)}

\newcommand{\ck}[1]{\norm{#1}_{W^{k,\infty}}}
\newcommand{\cki}[1]{\norm{#1}_{W^{k,\infty}(I_i^N)}}

\newcommand{\ckunit}[1]{\norm{#1}_{W^{k,\infty}(\Omega)}}

\newcommand{\hki}[1]{\norm{#1}_{H^k(I_i^N)}}
\newcommand{\hkunit}[1]{\norm{#1}_{H^k(\Omega)}}
\newcommand{\ltwo}[1]{\norm{#1}_{L^2(\Omega)}}
\newcommand{\htwo}[1]{\norm{#1}_{H^2(\Omega)}}

\newcommand{\uh}{\hat{u}^N}

\newcommand{\hp}{\hat{p}}

\newcommand{\pde}{\mathrm{PDE}}
\newcommand{\divv}{\mathrm{div}}
\newcommand{\inte}{\mathrm{int}}

\newcommand{\rpde}{\mathcal{R}_{\mathrm{PDE}}}
\newcommand{\rdiv}{\mathcal{R}_{\mathrm{div}}}
\newcommand{\rs}{\mathcal{R}_{s}}
\newcommand{\rsu}{\mathcal{R}_{s,u}}
\newcommand{\rsp}{\mathcal{R}_{s,p}}
\newcommand{\rsgu}{\mathcal{R}_{s,\nabla u}}
\newcommand{\rt}{\mathcal{R}_{t}}
\newcommand{\ru}{\mathcal{R}_{u}}
\newcommand{\rgraduab}{\mathcal{R}_{\nabla u}}
\newcommand{\rp}{\mathcal{R}_{p}}

\newcommand{\EL}{\mathcal{D}}

\newcommand{\hn}{\hat{n}}
\newcommand{\qu}[1]{\mathcal{Q}_{#1}}

\newtheorem{theorem}{Theorem}[section]

\newtheorem{remark}[theorem]{Remark}
\newtheorem{definition}[theorem]{Definition}
\newtheorem{lemma}[theorem]{Lemma}

\newtheorem{corollary}[theorem]{Corollary}

\begin{document}
\maketitle

\begin{abstract}
We prove rigorous bounds on the errors resulting from the approximation of the incompressible Navier-Stokes equations with (extended) physics-informed neural networks. We show that the underlying PDE residual can be made arbitrarily small for tanh neural networks with two hidden layers. Moreover, the total error can be estimated in terms of the training error, network size and number of quadrature points. The theory is illustrated with numerical experiments.   
\end{abstract}

\section{Introduction}
Deep learning has been very successfully deployed in a variety of fields including computer vision, natural language processing, game intelligence, robotics, augmented reality and autonomous systems \citep{DLnat} and references therein. In recent years, deep learning is being increasingly used in various contexts in scientific computing such as protein folding and controlled nuclear fusion. 

As deep neural networks are universal function approximators, it is also natural to use them as \emph{ansatz spaces} for the solutions of (partial) differential equations (PDEs). In fact, the literature on the use of deep learning for numerical approximation of PDEs has witnessed exponential growth in the last 2-3 years. Prominent examples for the use of deep learning in PDEs include the deep neural network approximation of high-dimensional semi-linear parabolic PDEs \citep{HEJ1}, linear elliptic PDEs \citep{SZ1,Kuty} and nonlinear hyperbolic PDEs \citep{LMR1,LMPR1} and references therein. More recently, DNN-inspired architectures such as DeepONets \citep{ChenChen,DeepOnet,LMK1} and Fourier neural operators \citep{FNO,KLM1} have been shown to even learn infinite-dimensional \emph{operators}, associated with underlying PDEs, efficiently.

Another extremely popular avenue for the use of machine learning in numerical approximation of PDEs is in the area of \emph{physics-informed neural networks} (PINNs). First proposed in slightly different forms in the 90s \citep{DPT,Lag1,Lag2}, PINNs were resurrected recently in \citep{KAR1,KAR2} as a practical and computationally efficient paradigm for solving both forward and inverse problems for PDEs. Since then, there has been an explosive growth in designing and applying PINNs for a variety of applications involving PDEs. A very incomplete list of references includes \citep{KAR4,KAR6,KAR7,KAR8,jagtap2020extended,jag2,jagtap2022physics, jin2021nsfnets,MM1,MM2,MM3,BKMM1, shukla2021physics, jagtap2022deep, hu2021extended, shukla2021parallel} and references therein.  

On the other hand and in stark contrast to the widespread applications of PINNs, there has been a pronounced scarcity of papers that rigorously justify why PINNs work. Notable exceptions include \citep{shin2020convergence} where the authors show consistency of PINNs with the underlying linear elliptic and parabolic PDE under stringent assumptions and in \citep{Zhang1} where similar estimates are derived for linear advection equations. In \citep{MM1,MM2}, the authors proposed a strategy for deriving error estimates for PINNs. To describe this strategy and highlight the underlying theoretical issues, it is imperative to introduce PINNs and we do so in an informal manner here (see section \ref{sec:preliminaries} for the formal definitions). To this end, we consider the following very general form of an abstract PDE,
\begin{equation}
\label{eq:PDE}
\begin{aligned}
    \EL[u](x,t) &= 0, \quad \mathcal{B}u(y,t) = \psi(y,t), \\
    u(x,0) &= \varphi(x), \quad \text{for } x\in D, y\in \partial D, t\in [0,T].
    \end{aligned}
\end{equation}
Here, $D\subset \mathbb{R}^d$ is compact and $\EL,\mathcal{B}$ are the differential and boundary operators, $u:D\times [0,T]\to\mathbb{R}^m$ is the solution of the PDE, $\psi:\partial D\times [0,T]\to\mathbb{R}^m$ specifies the (spatial) boundary condition and $\varphi:D\to\mathbb{R}^m$ is the initial condition.

We seek deep neural networks $u_{\theta}:D\times [0,T]\to\mathbb{R}^m$ (see \eqref{eq:dnn} for a definition), parameterized by $\theta \in \Theta$, constituting the weights and biases, that approximate the solution $u$ of \eqref{eq:PDE}. The key idea behind PINNs is to consider pointwise \emph{residuals}, defined for any sufficiently smooth function $f:D\times [0,T]\to\mathbb{R}^m$ as, 
\begin{equation}
\begin{aligned}
    \mathcal{R}_i[f](x,t) &= \EL[f](x,t), \quad   \mathcal{R}_s[f](y,t) = \mathcal{B}f(y,t) - \psi(y,t), \\ 
    \mathcal{R}_t[f](x) &= f(x,0) - \varphi(x), \quad x\in D, y\in \partial D, t\in [0,T]
    \end{aligned}
\end{equation}
for $x\in D$, $y\in \partial D$, $t\in [0,T]$. Using these residuals, one measures how well a function $f$ satisfies resp. the PDE, the boundary condition and the initial condition of \eqref{eq:PDE}. Note that for the exact solution $ \mathcal{R}_i[u]=\mathcal{R}_s[u]=\mathcal{R}_t[u]=0$. 

Hence, within the PINNs algorithm, one seeks to find a neural network $u_\theta$, for which all residuals are simultaneously minimized, e.g. by minimizing the quantity,
\begin{equation}
\label{eq:def-generalization-error}
\begin{aligned}
    \Eg(\theta)^2 &= \int_{D\times[0,T]} \abs{\mathcal{R}_i[u_\theta](x,t)}^2 dxdt +  \int_{\partial D\times[0,T]} \abs{\mathcal{R}_s[u_\theta](x,t)}^2 ds(x)dt \\
    &+  \int_{D} \abs{\mathcal{R}_t[u_\theta](x)}^2 dx. 
    \end{aligned}
\end{equation}However, the quantity $\Eg(\theta)$, often referred to as the \emph{population risk} or  \textit{generalization error} \citep{MM1} of the neural network $u_\theta$ involves integrals and can therefore not be directly minimized in practice. Instead, the integrals in \eqref{eq:def-generalization-error} are approximated by a suitable numerical quadrature (see section \ref{sec:quad} for details), resulting in,
\begin{equation}
\label{eq:et}
\begin{aligned}
    \Et^i(\theta,\S_i)^2 &= \sum_{n=1}^{N_i} w^n_i \abs{\mathcal{R}_i[u_\theta](x^n_i,t^n_i)}^2, \\
      \Et^s(\theta,\S_s)^2 &= \sum_{n=1}^{N_s} w^n_s \abs{\mathcal{R}_s[u_\theta](x^n_s,t^n_s)}^2, \quad 
      \Et^t(\theta,\S_t)^2 &= \sum_{n=1}^{N_t} w^n_t \abs{\mathcal{R}_t[u_\theta](x^t_i)}^2, \\
      \Et(\theta,\S)^2 &= \Et^i(\theta,\S_i)^2 + \Et^s(\theta,\S_s)^2+ \Et^t(\theta,\S_t)^2, 
\end{aligned}
\end{equation}
with quadrature points in space-time constituting data sets $\S_i = \{(x^n_i, t^n_i)\}_{n=1}^{N_i}$, $\S_s = \{(x^n_s, t^n_s)\}_{n=1}^{N_s}$ and $\S_t = \{x^n_t\}_{n=1}^{N_t}$, and $w^n_q$ are suitable quadrature weights for $q=i,t,s$. 

Thus, the underlying essence of PINNs is to minimize the training error $ \Et(\theta,\S)^2$ over the neural network parameters $\theta$. This procedure immediately raises the following key theoretical questions (see also \citep{DRM1}) starting with 
\begin{itemize}
    \item [[Q1.]] Given a tolerance $\varepsilon > 0$, do there exist neural networks $\hat{u} = u_{\hat{\theta}}$, $\widetilde{u} = u_{\widetilde{\theta}}$, parametrized by $\hat{\theta}, \widetilde{\theta} \in \Theta$ such that the corresponding generalization $\Eg(\hat{\theta})$ \eqref{eq:def-generalization-error} and training $\Et(\widetilde{\theta},\S_s)$\eqref{eq:et} errors are small i.e., $\Eg(\hat{\theta}),\Et(\widetilde{\theta},\S_s) < \varepsilon$?
\end{itemize}    
As the aim in the PINNs algorithm is to minimize the training error (and indirectly the generalization error), an affirmative answer to this question is of vital importance as it ensures that the loss (PDE residual) being minimized can be made small. However, minimizing the PDE residual does not necessarily imply that the overall error (difference between the exact solution of the PDE \eqref{eq:PDE} and its PINN approximation) is small. This leads to the second key question, 
\begin{itemize}
    \item [[Q2.]] Given a PINN $\hat{u}$ with small generalization error, is the corresponding \emph{total error} $\|u-\hat{u}\|$ small, i.e., is $\|u-\hat{u}\|< \delta (\varepsilon)$, for some $\delta(\varepsilon) \sim {\mathcal O}(\varepsilon)$, for some suitable norm $\|\cdot\|$, and with $u$ being the solution of the PDE \eqref{eq:PDE}?
    
\end{itemize}
An affirmative answer to Q2 (and Q1) certifies that, \emph{in principle}, there exists a (physics-informed) neural network, corresponding to the parameter $\hat{\theta}$, such that the PDE residual and consequently, the overall error in approximating the solution of the PDE \eqref{eq:PDE}, are small. However, in practice, we minimize the training error $\Et$ \eqref{eq:et} and this leads to another key question, 
\begin{itemize}
    \item [[Q3.]] Given a small training error $\Et(\theta^{\ast})$ and a sufficiently large training set $\S$, is the corresponding generalization error $\Eg(\theta^{\ast})$ also proportionately small?
\end{itemize}
An affirmative answer to question Q3, together with question Q2, will imply that the trained PINN $u_{\theta^*}$ is an accurate approximation of the solution $u$ of the underlying PDE \eqref{eq:PDE}. 
Thus, answering the above three questions affirmatively will constitute a comprehensive theoretical investigation of PINNs and provide a rationale for their very successful empirical performance. 

Given this context, we examine how far the literature has come in answering these key questions on the theory for PINNs. In \citep{MM1,MM2}, the authors leverage the \emph{stability} of solutions of the underlying PDE \eqref{eq:PDE} to bound the total error in terms of the generalization error (question Q2). Similarly, they use the accuracy of quadrature rules to bound the generalization error in terms of the training error (question Q3). This approach is implemented for forward problems corresponding to a variety of PDEs such as the semi-linear and quasi-linear parabolic equations and the incompressible Euler and the Navier-Stokes equations \citep{MM1}, radiative transfer equations \citep{MM3}, nonlinear dispersive PDEs such as the KdV equations \citep{BKMM1} and for the unique continuation (data assimilation) inverse problem for many linear elliptic, parabolic and hyperbolic PDEs \citep{MM2}. However, Q1 was not answered in these papers. Moreover, the authors imposed rather stringent assumptions on the weights and biases of the trained PINN, which may not hold in practice. 

In \citep{DRM1}, the authors answered the key questions Q1, Q2 and Q3 in the case of a large class of \emph{linear parabolic PDEs}, namely the Kolmogorov PDEs, which include the heat equation and the Black-Scholes equation of option pricing as special examples. Thus, they provided a rigorous and comprehensive error analysis of PINNs for these PDEs. Moreover, they also showed that PINNs overcome the \emph{curse of dimensionality} in the context of very high-dimensional Kolmogorov equations. 

The authors of \citep{DRM1} utilized the linearity of the underlying Kolmogorov heavily in their analysis. It is natural to ask if analogous error estimates can be shown for PINN approximations of nonlinear PDEs. This consideration sets the stage for the current paper where we carry out a thorough error analysis for PINNs approximating a prototypical nonlinear PDE and answer Q1, Q2 and Q3 affirmatively. The nonlinear PDE that we consider is the incompressible Navier-Stokes equation, which is the fundamental mathematical model governing the flow of incompressible Newtonian fluids \citep{temam2001navier}. 

We are going to show the following results on the PINN approximation of the incompressible Navier-Stokes equations,
\begin{itemize}
    \item We show that there exist neural networks that approximate the classical solutions of Navier-Stokes equations such that the PINN generalization error \eqref{eq:def-generalization-error} and the PINN training error \eqref{eq:et} can be made arbitrarily small. Moreover, we provide explicit bounds on the number of neurons as well as the weights of the network in terms of error tolerance and Sobolev norms of the underlying Navier-Stokes equations. This analysis is also extended for the XPINN approximation \citep{jagtap2020extended} of the Navier-Stokes equations, answering Q1 affirmatively for both PINNs and XPINNs. 
    \item We bound the total error of the PINN (and XPINN) approximation of the Navier-Stokes equations in terms of the PDE residual (generalization error \eqref{eq:def-generalization-error}). Consequently, a small PDE residual implies a small total error, answering Q2 affirmatively. \item We bound the generalization error \eqref{eq:def-generalization-error} in terms of the training error \eqref{eq:et} and the number of quadrature points using a midpoint quadrature rule. This affirmatively answers question Q3 and establishes the fact that a small training error and sufficient number of quadrature points suffice to yield a small total error for the PINN (and XPINN) approximation of the Navier-Stokes equations, under a mild assumption on the growth of the network weights (see discussion in Section \ref{sec:33}).  
    \item We present numerical experiments to illustrate our theoretical results. 
    
\end{itemize}

The rest of our paper is organized as follows: In section \ref{sec:preliminaries}, we collect preliminary information on the Navier-Stokes equations and neural networks and present the PINN and XPINN algorithms. The error analysis is carried out in section \ref{sec:3} and numerical experiments are presented in section \ref{sec:4}. 

\section{Preliminaries}\label{sec:preliminaries}
In this section, we collect preliminary information on concepts used in rest of the paper. We start with the form of the Navier-Stokes equations. 
\subsection{The incompressible Navier-Stokes equations}
\label{sec:NS}
We consider the well-known incompressible Navier-Stokes equations \citep{temam2001navier} and references therein,
\begin{equation}\label{eq:navier-stokes}
   \begin{cases} u_t + u\cdot \nabla u + \nabla p = \nu
  \Delta u  &\text{in } D \times [0,T] ,\\
  \div{u} = 0 &\text{in } D \times [0,T],\\
  u(t=0) = u_0 &\text{in }  D.
  \end{cases}
\end{equation}
Here, $u:D\times [0,T] \to \mathbb{R}^d$ is the fluid velocity, $p:D\to\mathbb{R}$ is the pressure and $u_0:D \to \mathbb{R}^d$ is the initial fluid velocity. The viscosity is denoted by $\nu\geq 0$. 
For the rest of the paper, we consider the Navier-Stokes equations \eqref{eq:navier-stokes} on the $d$-dimensional torus $D = \mathbb{T}^d = [0,1)^d$ with periodic boundary conditions. 

The existence and regularity of the solution to  \eqref{eq:navier-stokes} depends on the regularity of $u_0$, as is stated by the following well-known theorem \cite[Theorem 3.4]{majda2002vorticity}. Other regularity results with different boundary conditions can be found in e.g. \citep{temam2001navier}.

\begin{theorem}\label{thm:NS-existence}
If $u_0\in H^r(\mathbb{T}^d)$ with $r>\frac{d}{2}+2$ and $\div{u_0}=0$, then there exist $T>0$ and a classical solution $u$ to the Navier-Stokes equation such that $u(t=0)=u_0$ and $u\in C([0,T];H^r(\mathbb{T}^d))\cap C^1([0,T];H^{r-2}(\mathbb{T}^d))$.
\end{theorem}

Based on this result, we prove that $u$ is Sobolev regular i.e., that $u\in H^k(D\times [0,T])$ for some $k\in\mathbb{N}$, provided that $r$ is large enough. 

\begin{corollary}\label{cor:NS-Hk}
If $k\in\mathbb{N}$ and $u_0\in H^r(\mathbb{T}^d)$ with $r>\frac{d}{2}+2k$ and $\div{u_0}=0$, then there exist $T>0$ and a classical solution $u$ to the Navier-Stokes equation such that $u\in H^{k}(\mathbb{T}^d\times [0,T])$, $\nabla p\in H^{k-1}(\mathbb{T}^d\times [0,T])$ and $u(t=0)=u_0$. 
\end{corollary}

\begin{proof}
The corollary follows directly from Theorem \ref{thm:NS-existence} for $k=1$. Therefore, let $k\geq 2$ be arbitary and assume that $r>\frac{d}{2}+2k$ and $u_0\in H^r(\mathbb{T}^d)$ and $\div{u_0}=0$. By Theorem \ref{thm:NS-existence} there exists $T>0$ and a classical solution $u$ to the Navier-Stokes equation such that $u(t=0)=u_0$ and $u\in C([0,T];H^r(\mathbb{T}^d))\cap C^1([0,T];H^{r-2}(\mathbb{T}^d))$. Following \cite[Section 1.8]{majda2002vorticity}, we find that the pressure $p$ satisfies the equation
\begin{equation}\label{eq:poisson-p}
    -\Delta p = \mathrm{Trace}((\nabla u)^2) = \sum_{i,j} u^i_{x_j}u^j_{x_i}. 
\end{equation}
As $r>\frac{d}{2}+1$, $H^{r-1}(\mathbb{T}^d)$ is a Banach algebra (see Lemma \ref{lem:banach-algebra}), it holds that $\Delta p\in C([0,T];H^{r-1}(\mathbb{T}^d))$ and accordingly $\nabla p\in C([0,T];H^{r}(\mathbb{T}^d))$. Since $u\in C^1([0,T];H^{r-2})$, we can take the time derivative of equation \eqref{eq:poisson-p} to find that $\Delta p_t \in C([0,T];H^{r-3})$, since the conditions for $H^{r-3}(\mathbb{T}^d)$ to be a Banach algebra are met. As a result we find that $\nabla p_t \in C([0,T];H^{r-2})$. Taking the time derivative of the Navier-Stokes equations \eqref{eq:navier-stokes}, we find that $u_{tt} \in C([0,T];H^{r-4})$ and therefore $u\in  C^2([0,T];H^{r-4})$. Repeating these steps, one can prove that $u\in \cap_{\ell=0}^kC^\ell([0,T];H^{r-2\ell}(\mathbb{T}^d))$. The statement of the corollary then follows from this observation since $\ell+r-2\ell \geq k$ for all $0\leq \ell \leq k$ if $r> \frac{d}{2}+ 2k$. Similarly, one can prove that $\nabla p\in \cap_{\ell=0}^{k-1}C^\ell([0,T];H^{r-2\ell}(\mathbb{T}^d))$. 
\end{proof}

\subsection{Neural networks}
As our objective is to approximate the solution of the incompressible Navier-Stokes equations \eqref{eq:navier-stokes} with neural networks, here we formally introduce our definition of a neural network and the related terminology. 
\begin{definition}
\label{def:nn}
Let $R\in(0,\infty]$, $L,W\in\mathbb{N}$ and $l_0,\ldots, l_L\in\mathbb{N}$. Let $\sigma:\mathbb{R}\to\mathbb{R}$ be a twice differentiable \emph{activation function} and define 
\begin{equation}
   \Theta =  \Theta_{L,W,R} := \bigcup_{L'\in\mathbb{N}, L'\leq L}\:\bigcup_{l_0,\ldots,l_L\in\{1, \ldots, W\}}\bigtimes_{k=1}^{L'} \left([-R,R]^{l_k\times l_{k-1}}\times[-R,R]^{l_k}\right). 
\end{equation}
For $\theta\in\Theta_{L,W,R}$,  we define $\theta_k:= (\mathcal{W}_k,b_k)$ and $\mathcal{A}_k:\mathbb{R}^{l_{k-1}}\to\mathbb{R}^{l_{k}}:x\mapsto \mathcal{W}_k x+b_k$ for $1\leq k\leq L$ and and we define $f^\theta_k:\mathbb{R}^{l_{k-1}}\to\mathbb{R}^{l_{k}}$ by
\begin{equation}
    f_k^\theta(z) = \begin{cases}\mathcal{A}_L^\theta(z) & k=L,\\ (\sigma \circ \mathcal{A}_{k}^\theta)(z) & 1\leq k < L. \end{cases}
\end{equation}
We denote by $u_\theta:\mathbb{R}^{l_0}\to\mathbb{R}^{l_L}$ the function that satisfies for all $z \in\mathbb{R}^{l_0}$ that
\begin{equation}
\label{eq:dnn}
     u_\theta(z) = \left(f_{L}^\theta\circ f_{L-1}^\theta \circ \cdots \circ f_1^\theta\right)(z), 
\end{equation}
where in the setting of approximating the Navier-Stokes equation \eqref{eq:navier-stokes} we set $l_0 = d+1$ and $z=(x,t)$.
We refer to $u_\theta$ as the realization of the \emph{neural network} associated to the parameter $\theta$ with $L$ layers and widths $(l_0,l_1, \ldots, l_L)$. We refer to the first $L-1$ layers as \emph{hidden layers}. For $1\leq k\leq L$, we say that layer $k$ has width $l_k$ and we refer to $\mathcal{W}_k$ and $b_k$ as the \emph{weights and biases} corresponding to layer $k$. {\color{black} The width of $u_\theta$ is defined as $\max(l_0,\dots, l_L)$.} If $L=2$, we say that $u_\theta$ is a \emph{shallow neural network}; if $L\geq 3$, we say that $u_\theta$ is a \emph{deep neural network}. 
\end{definition}

\subsection{Quadrature rules}\label{sec:quad}

In the following sections, we will need to approximate integrals of functions. For this reason, we introduce some notation and recall well-known results on numerical quadrature rules. 

Given $\Lambda \subset \mathbb{R}^d$ and $f\in L^1(\Lambda)$, we will be interested in approximating $\int_\Lambda f(y)dy$, with $dy$ denoting the $d$-dimensional Lebesgue measure. A numerical quadrature rule provides such an approximation by choosing some quadrature points $y_m\in \Lambda$ for $1\leq m\leq M$, and quadrature weights $w_m>0$ for $1\leq m\leq M$, and considers the approximation 
\begin{equation}
    \frac{1}{M}\sum_{m=1}^M w_m f(y_m) \approx  \int_\Lambda f(y)dy. 
\end{equation}
The accuracy of this approximation depends on the chosen quadrature rule, the number of quadrature points $M$ and the regularity of $f$. Whereas in very high dimensions, random training points or low-discrepancy training points \citep{mishra2020enhancing} are needed, the relatively low-dimensional setting of the Navier-Stokes equations i.e., $d\leq 4$, allows the use of standard deterministic numerical quadrature points. In order to obtain explicit rates, we will focus on the \emph{midpoint rule}, but our analysis will also hold for general deterministic numerical quadrature rules.

We briefly recall the midpoint rule. For $N\in\mathbb{N}$, we partition $\Lambda$ into $M \sim N^d$ cubes of edge length $1/N$ and we denote by $\{y_m\}_{m=1}^M$ the midpoints of these cubes. The formula and accuracy of the midpoint rule $\qu{M}^\Lambda$ are then given by, 
\begin{equation}\label{eq:quad-error}
    \qu{M}^\Lambda[f] := \frac{1}{M}\sum_{m=1}^Mf(y_m), \qquad \abs{\int_\Lambda f(y) dy - \qu{M}^\Lambda[f]} \leq C_f M^{-2/d},
\end{equation}
where $C_f \lesssim \norm{f}_{C^2}$. 

\subsection{Physics-informed neural networks (PINNs)}\label{sec:PINNs}

We seek deep neural networks $u_{\theta}:D\times [0,T]\to\mathbb{R}^d$ and $p_{\theta}:D\times [0,T]\to\mathbb{R}$ (cf. Definition \ref{def:nn}), parameterized by $\theta \in \Theta$, constituting the weights and biases, that approximate the solution $u$ of \eqref{eq:navier-stokes}. To this end, the key idea behind PINNs is to consider pointwise \emph{residuals}, defined in the setting of the Navier-Stokes equations \eqref{eq:navier-stokes} for any sufficiently smooth $v:D\times [0,T]\to \mathbb{R}^{d}$ and $q:D\times [0,T]\to \mathbb{R}$ as, 
\begin{align}\label{eq:pinn-residuals}
\begin{split}
    \rpde[(v,q)](x,t) &= (v_t + v\cdot \nabla v + \nabla q - \nu
  \Delta v)(x,t), \\   
  \rdiv[v](x,t) &= \div{v}(x,t)\\
  \rs[v](y,t) &= \mathcal{B}v(y,t) - \psi(y,t), 
  \\ \rt[v](x) &= v(x,0) - \varphi(x)
\end{split}
\end{align}
for $x\in D$, $y\in \partial D$, $t\in [0,T]$. In the above, $\mathcal{B}$ is the boundary operator, $\psi:\partial D\times [0,T]\to\mathbb{R}^d$ specifies the (spatial) boundary condition and $\varphi:D\to\mathbb{R}^d$ is the initial condition.
Using these residuals, one measures how well a function $f$ satisfies resp. the PDE, the boundary condition and the initial condition of \eqref{eq:navier-stokes}. Note that for the exact solution to the Navier-Stokes equations \eqref{eq:navier-stokes} it holds that $  \rpde[(u,p)]=\rdiv[u]=\rs[u]=\rt[u]=0$. 

Hence, within the PINNs algorithm, one seeks to find a neural network $(u_\theta,p_\theta)$, for which all residuals are simultaneously minimized, e.g. by minimizing the quantity,
\begin{align}\label{eq:generalization-error-pinn}
\begin{split}
     \Eg(\theta)^2 =&\: \int_{D\times[0,T]} \norm{\rpde[(u_\theta,p_\theta)](x,t)}_{\mathbb{R}^d}^2 dxdt + \int_{D\times[0,T]} \abs{\rdiv[u_\theta](x,t)}^2 dxdt\\&+  \int_{\partial D\times[0,T]} \norm{\rs[u_\theta](x,t)}_{\mathbb{R}^d}^2 ds(x)dt +  \int_{D} \norm{\rt[u_\theta](x)}_{\mathbb{R}^d}^2 dx. 
\end{split}
\end{align}
The different terms of \eqref{eq:generalization-error-pinn} are often rescaled using some weights. For simplicity, we set all these weights to one. 
The quantity $\Eg(\theta)$, often referred to as the \emph{population risk} or  \textit{generalization error} of the neural network $u_\theta$, involves integrals and can therefore not be directly minimized in practice. Instead, the integrals in \eqref{eq:generalization-error-pinn} are approximated by a numerical quadrature, as introduced in Section \ref{sec:quad}. As a result, we define the (squared) training loss for PINNs $\theta \mapsto \Et(\theta, \S)^2$ as follows, 
\begin{equation}
\label{eq:training-loss-pinn}
\begin{aligned}
\Et(\theta,\S)^2 =&\: \Et^\pde(\theta,\S_\inte)^2 +  \Et^\divv(\theta,\S_\inte)^2 +\Et^s(\theta,\S_s)^2+ \Et^t(\theta,\S_t)^2\\
     =&\: \sum_{n=1}^{N_\inte} w^n_{\inte}\norm{\rpde[(u_\theta,p_\theta)](t^n_\inte,x^n_\inte)}_{\mathbb{R}^d}^2
     +\sum_{n=1}^{N_\inte} w^n_{\inte}\abs{\rdiv[u_\theta](t^n_\inte,x^n_\inte)}^2 \\
      &+ \sum_{n=1}^{N_s} w^n_s\norm{\mathcal{R}_s[u_\theta](t^n_s,x^n_s)}_{\mathbb{R}^d}^2 +
      \sum_{n=1}^{N_t} w^n_t\norm{\mathcal{R}_t[u_\theta](x^n_t)}_{\mathbb{R}^d}^2,
      \end{aligned}
\end{equation}
where the training data set $\S = (\S_\inte, \S_s, \S_t)$ is chosen as quadrature points with respect to the relevant domain (resp. $D\times [0,T]$, $\partial D\times [0,T]$ and $D$) and where the $w^n_\ast$ are corresponding quadrature weights.  

A \emph{trained PINN} $u^{\ast} = u_{\theta^{\ast}}$ is then defined as a (local) minimum of the optimization problem, 
\begin{equation}
\label{eq:opt}
    \theta^*(\S) = \arg\min_{\theta\in\Theta}  \Et(\theta,\S)^2, 
\end{equation}
with loss function \eqref{eq:training-loss-pinn} (possibly with additional data and weight regularization terms), found by a (stochastic) gradient descent algorithm such as ADAM or L-BFGS. 

\subsection{Extended physics-informed neural networks (XPINNs)}
\label{sec:XPINNs}
In many applications, it happens that the computational domain has a very complicated shape or that the PDE solution shows localized features. In such cases, it is beneficial to decompose the computational domain into non-overlapping regions and deploy different neural networks to approximate the PDE solution in different sub-regions. This idea was first presented in \citep{jagtap2020extended}, where the authors proposed to decompose the domain in $\mathcal{N}$ closed subdomains with non-overlapping interior and deploy PINNs $u_{\theta_q}$ to approximate the exact solution $u$ in each of those subdomains $\Omega_q$.  Patching together the PINNs for all the subnetworks yields the final approximation $u_\theta$, termed \textit{extended physics-informed neural network (XPINN)}, defined as, 
\begin{equation}\label{eq:xpinn}
    u_\theta(z) = 
      \sum_{q=1}^\mathcal{N} \chi_q(z) u_{\theta_q}(z),
\end{equation}
for $z=(x,t)$ and where the weight function $\chi_q$ is given by,
\begin{equation}
   \chi_q(z) = \begin{cases}
      0 & z\not\in\Omega_q, \\
      \frac{1}{\#\{n\: : \: z \in \Omega_n\}} & z\in\Omega_q,  
    \end{cases}
\end{equation}
where $\#\{n\: : \: z \in \Omega_n\}$ represents the number of subdomains $z$ belongs to. Hence $\sum_{q}\chi_q(z) = 1$ for all $z$. One can define neural networks $p_\theta$ and $p_{\theta_q}$ in an analogous way. 
It is clear that mimimizing the standard PINN loss \eqref{eq:training-loss-pinn} for an XPINN \eqref{eq:xpinn} would not be a suitable approach. It is necessary that additional terms in the form of \textit{interface conditions} should be added to the loss function. For this purpose, we define for every $q$ the following residuals in addition to the standard PINN residuals \eqref{eq:pinn-residuals}, 
\begin{align}
    \begin{split}
        \ru[f](y,t) = f(y,t) - u_\theta(y,t),
    \end{split}
\end{align}
where $y\in \partial \Omega_q \setminus \partial D$ and $t\in [0,T]$. 
The squared generalization error of an XPINN $u_\theta$ is then given by
\begin{align}\label{eq:generalization-error-xpinn}
\begin{split}
     \Eg(\theta)^2 =&\: \int_{D\times[0,T]} \norm{\rpde[(u_\theta, p_\theta)](x,t)}_{\mathbb{R}^d}^2 dxdt + \int_{D\times[0,T]} \abs{\rdiv[u_\theta](x,t)}^2 dxdt\\&+  \int_{\partial D\times[0,T]} \norm{\rs[u_\theta](x,t)}_{\mathbb{R}^d}^2 ds(x)dt +  \int_{D} \norm{\rt[u_\theta](x)}_{\mathbb{R}^d}^2 dx\\
     &+ \sum_{q=1}^\mathcal{N}\int_{(\partial \Omega_q \setminus \partial D)\times[0,T]}\norm{\ru[u_{\theta_q}](x,t)}_{\mathbb{R}^d}^2 ds(x)dt\\
     &+ \sum_{q=1}^\mathcal{N}\int_{(\partial \Omega_q \setminus \partial D)\times[0,T]}\norm{\rpde[(u_{\theta_q}, p_{\theta_q})](x,t)-\rpde[u_{\theta}](x,t)}_{\mathbb{R}^d}^2 ds(x)dt
     . 
\end{split}
\end{align}
The interface conditions on the two last lines of \eqref{eq:generalization-error-xpinn} enforce the continuity and possibly even higher regularity of the XPINN at the interface of neighbouring subdomains. 

The XPINN training loss can then be defined by replacing the integrals in \eqref{eq:generalization-error-xpinn} by numerical quadratures, in the same way the standard PINN training loss \eqref{eq:training-loss-pinn} was derived from \eqref{eq:generalization-error-pinn}.

\section{Error analysis}
\label{sec:3}
In this section, we will obtain rigorous on the PINN and XPINN approximations of the solutions of the incompressible Navier-Stokes equations. We start with bounds on PINN residuals below. 

\subsection{Bound on the PINN residuals} From the definition of the interior PINN residuals \eqref{eq:pinn-residuals}, it is clear that if we can find a neural network $\hat{u}$ such that $\norm{u-\hat{u}}_{H^2(D\times [0,T])}$ is small, then the interior PINN residual will be small as well. The approximation (in Sobolev norm) of Sobolev regular functions by tanh neural networks is discussed in Appendix \ref{sec:tanh}. The main ingredients are a piecewise polynomial approximation, the existence of which is guaranteed by the Bramble-Hilbert lemma, and the ability of tanh neural networks to efficiently approximate polynomials, the multiplication operator and an approximate partition of unity. The main result of Appendix \ref{sec:tanh} is Theorem \ref{thm:tanh-approximation}, which is a variant of \cite[Theorem 5.1]{deryck2021approximation}. It proves that a tanh neural network with two hidden layers suffices to make $\norm{u-\hat{u}}_{H^2(D\times [0,T])}$ arbitrarily small and provides explicit bounds on the needed network width. Using this theorem, we can prove the following upper bound on the PINN residual. 

\begin{theorem}\label{thm:pinn-approx}
Let $n\geq 2$, $d,r,k\in\mathbb{N}$, with $k\geq 3$,
and let $u_0\in H^r(\mathbb{T}^d)$ with $r>\frac{d}{2}+2k$ and $\div{u_0}=0$. It holds that: 
\begin{itemize}
    \item there exist $T>0$ and a classical solution $u$ to the Navier-Stokes equations such that $u\in H^{k}(\Omega)$, $\nabla p\in H^{k-1}(\Omega)$, $\Omega = \mathbb{T}^d\times [0,T]$, and $u(t=0)=u_0$,  
    \item for every $N>5$, there exist tanh neural networks $\hu_j$, $1\leq j\leq d$, and $\widehat{p}$, each with two hidden layers, of widths $3\left\lceil\frac{k+n-2}{2}\right\rceil\binom{d+k-1}{d}+\lceil TN\rceil +dN$ and $3\left\lceil\frac{d+n}{2}\right\rceil \binom{2d+1}{d}\lceil TN \rceil N^{d}$, such that for every $1\leq j\leq d$, 
    \begin{align}
    \begin{split}
        \ltwo{(\hu_j)_t + \hu\cdot \nabla \hu_j + (\nabla\widehat{p})_j - \nu \Delta \hu_j} &\leq C_1\ln^2(\beta N) N^{-k+2},
    \end{split}\label{eq:bound-pinn-res-1}\\
    \ltwo{\div{\hu}} &\leq C_2 \ln(\beta N) N^{-k+1},\label{eq:bound-pinn-res-2}\\
     \norm{(u_0)_j-\hu_j(t=0)}_{L^2(\mathbb{T}^d)} &\leq C_3 \ln(\beta N)N^{-k+1}, \label{eq:bound-pinn-res-3}
    \end{align}
where the constants $\beta,C_1,C_2,C_3$ are explicitly defined in the proof and can depend on $k$, $d$, $T$, $u$ and $p$ but not on $N$. {\color{black}The weights of the networks can be bounded by $\bigO(N^\gamma \ln(N))$ where $\gamma = \max\{1,d(2+k^2+d)/n\}$.}
\end{itemize}

\end{theorem}

\begin{proof}
Let $N>5$. By Corollary \ref{cor:NS-Hk} it holds that $u\in H^k(\mathbb{T}^d\times [0,T])$ and $\nabla p\in H^{k-1}(\mathbb{T}^d\times [0,T])$, hence also $p\in H^{k-1}(\mathbb{T}^d\times [0,T])$. As a result of Theorem \ref{thm:tanh-approximation}, there then exists for every $1\leq j\leq d$ a tanh neural network $\hu_j:=\uh_j$ with two hidden layers and widths $3\left\lceil\frac{k+n-2}{2}\right\rceil\binom{d+k-1}{d}+\lceil TN\rceil +dN$ and $3\left\lceil\frac{d+n}{2}\right\rceil \binom{2d+1}{d}\lceil TN \rceil N^{d}$ such that for every $0\leq \ell \leq 2$,
\begin{equation}
    \norm{u_j-\hu_j}_{H^\ell(\Omega)} \leq C_{\ell,k,d+1, u_j} \lambda_{\ell}(N) N^{-k+\ell}, 
\end{equation}
where $\lambda_{\ell}(N) = 2^{\ell+1} 3^d \left(1+\delta\right)\ln^\ell\left(\beta_{\ell, d+1,u_j}N^{d+k+2}\right)$, $\delta = \frac{1}{100}$, and the definition of the other constants can be found in Theorem \ref{thm:tanh-approximation}. The weights can be bounded by $\bigO(N^\gamma \ln(N))$ where $\gamma = \max\{1,d(2+k^2+d)/n\}$. We write $\hu = (\hu_1, \ldots, \hu_d)$. Moreover, by Theorem \ref{thm:tanh-approximation}, there also exists a tanh neural network $\widehat{ p}:=\widehat{p}^N$ with two hidden layers and the same widths as before such that
\begin{equation}
    \ltwo{(\nabla p)_j-(\nabla\widehat{p})_j}\leq  \norm{p-\widehat{p}}_{H^1(\Omega)} \leq C_{1,k-1,d+1,p}\lambda_1(N) N^{-k+2}.
\end{equation}
It is now straightforward to bound the PINN residual. 
\begin{equation}
    \ltwo{(u_j)_t-(\hu_j)_t} \leq \abs{u_j-\hu_j}_{H^1(\Omega)} .
\end{equation}
By the Sobolev embedding theorem (Lemma \ref{lem:sobolev-embedding}) it follows from $u\in C^1([0,T], H^{r-2}(\mathbb{T}^d))$ that $u\in C^1(\Omega)$, and hence
\begin{align}
\begin{split}
     \ltwo{u\cdot \nabla u_j -\hu\cdot \nabla \hu_j} &\leq  \ltwo{u\cdot \nabla u_j -\hu\cdot \nabla u_j} +  \ltwo{\hu\cdot \nabla u_j -\hu\cdot \nabla \hu_j} \\
     &\leq  \sqrt{d}\norm{u_j}_{C^1}\max_i\ltwo{u_i-\hu_i} +   \sqrt{d}\max_i\norm{\hu_i}_{C^0}\abs{u_j-\hu_j}_{H^1(\Omega)}
\end{split}\end{align}
and finally also
\begin{align}
     \ltwo{\Delta u_j - \Delta \hu_j} &\leq \sqrt{d} \htwo{u_j -\hu_j}\\
     \ltwo{\div{u} - \div{\hu}} &\leq \sqrt{d} \max_i\abs{u_i -\hu_i}_{H^1(\Omega)}.
\end{align}
Hence, we find that for $1\leq j\leq d$,
\begin{align}
\begin{split}
    &\ltwo{(\hu_j)_t + \hu\cdot \nabla \hu_j + (\nabla\widehat{p})_j - \nu \Delta \hu_j} \leq C_{1,k-1,d+1,p}\lambda_1(N) N^{-k+2}\\
    &\quad  +C_{1,k,d+1, u_j}\lambda_{1}(N)(1+\sqrt{d}\max_i\norm{\hu_i}_{C^0})N^{-k+1} \\&\quad+ \sqrt{d}\lambda_{0}(N)\norm{u_j}_{C^1}C_{0,k,d+1, u_j}N^{-k} + \nu\sqrt{d}C_{2,k,d+1, u_j} \lambda_{2}(N)N^{-k+2}
\end{split}
\end{align}
and also
\begin{equation}
    \ltwo{\div{\hu}} \leq \sqrt{d}C_{1,k,d+1, u_1} \lambda_{1}(N)N^{-k+1}.
\end{equation}
Finally, we find from the multiplicative trace theorem (Lemma \ref{lem:trace-inequality}) that
\begin{align}
\begin{split}
    \norm{(u_0)_j-\hu_j(t=0)}_{L^2(\mathbb{T}^d)} &\leq \norm{u_j-\hu_j}_{L^2(\partial \Omega)} \\&\leq \sqrt{\frac{2\max\left\{2h_\Omega,d+1\right\}}{\rho_\Omega}}\norm{u_j-\hu_j}_{H^1(\Omega)}\\
    &\leq \sqrt{\frac{2\max\left\{2h_\Omega,d+1\right\}}{\rho_\Omega}}C_{1,k,d+1, u_1} \lambda_{1}(N)N^{-k+1}, 
\end{split}\end{align}
where $h_\Omega$ is the diameter of $\Omega$ and $\rho_\Omega$ is the radius of the largest $(d+1)$-dimensional ball that can be inscribed into $\Omega$. This concludes the proof.
\end{proof}

Thus, the bounds \eqref{eq:bound-pinn-res-1}, \eqref{eq:bound-pinn-res-2} and \eqref{eq:bound-pinn-res-3} clearly show that by choosing $N$ sufficiently large, we can make the PINN residuals \eqref{eq:pinn-residuals} and consequently the generalization error arbitrarily small. This affirmatively answers Q1 in the introduction. 

To further illustrate the bounds of Theorem \ref{thm:pinn-approx}, we look for a suitable neural network such that \eqref{eq:bound-pinn-res-1}, \eqref{eq:bound-pinn-res-2} and \eqref{eq:bound-pinn-res-3} are all smaller than $1\%$. Using the notation of the proof, we set $n=2$, $T=1$ and $\nu = \frac{1}{1000}$ and make the simplification that $ \hkunit{u} = 1$ for all $k$. The results are shown for $d=2,3$ in Figure \ref{fig:nn-size-pinn} for varying regularity $r$ of the initial condition i.e., $u_0\in H^r(\mathbb{T}^d)$. In particular, for $d=2$ we find that the minimal network size of every sub-network $\hu_j$ is $54\cdot 10^3$ neurons. Although it is certainly possible to reach this level of accuracy with smaller networks, see e.g. \citep{jin2021nsfnets}, the networks that follow from Theorem \ref{thm:pinn-approx} are not unreasonably large, even in three space dimensions.  

\begin{remark}
One can easily prove that the XPINN loss of the network constructed in the proof of Theorem \ref{thm:pinn-approx} will be small as well. 
\end{remark}
{\color{black}
\begin{remark}
In Theorem \ref{thm:pinn-approx} the parameter $n \geq 2$ can be chosen arbitrarily and is independent of $u$. It controls the trade-off between the network width (which grows linearly in $n$) and the network weights (which grow as $\bigO(N^\gamma \ln(N))$ where $\gamma = \max\{1,d(2+k^2+d)/n\}$). Note that it only makes sense to choose $2\leq n\leq d(2+k^2+d)$ as the bound on the weights can not be made smaller than $\bigO(N\ln(N))$. Hence, this bound proves the existence of a neural network architecture for which the weights will only grow very moderately with increasing accuracy. 
\end{remark}
}
\begin{figure}
    \centering
    \includegraphics[width=0.8\textwidth]{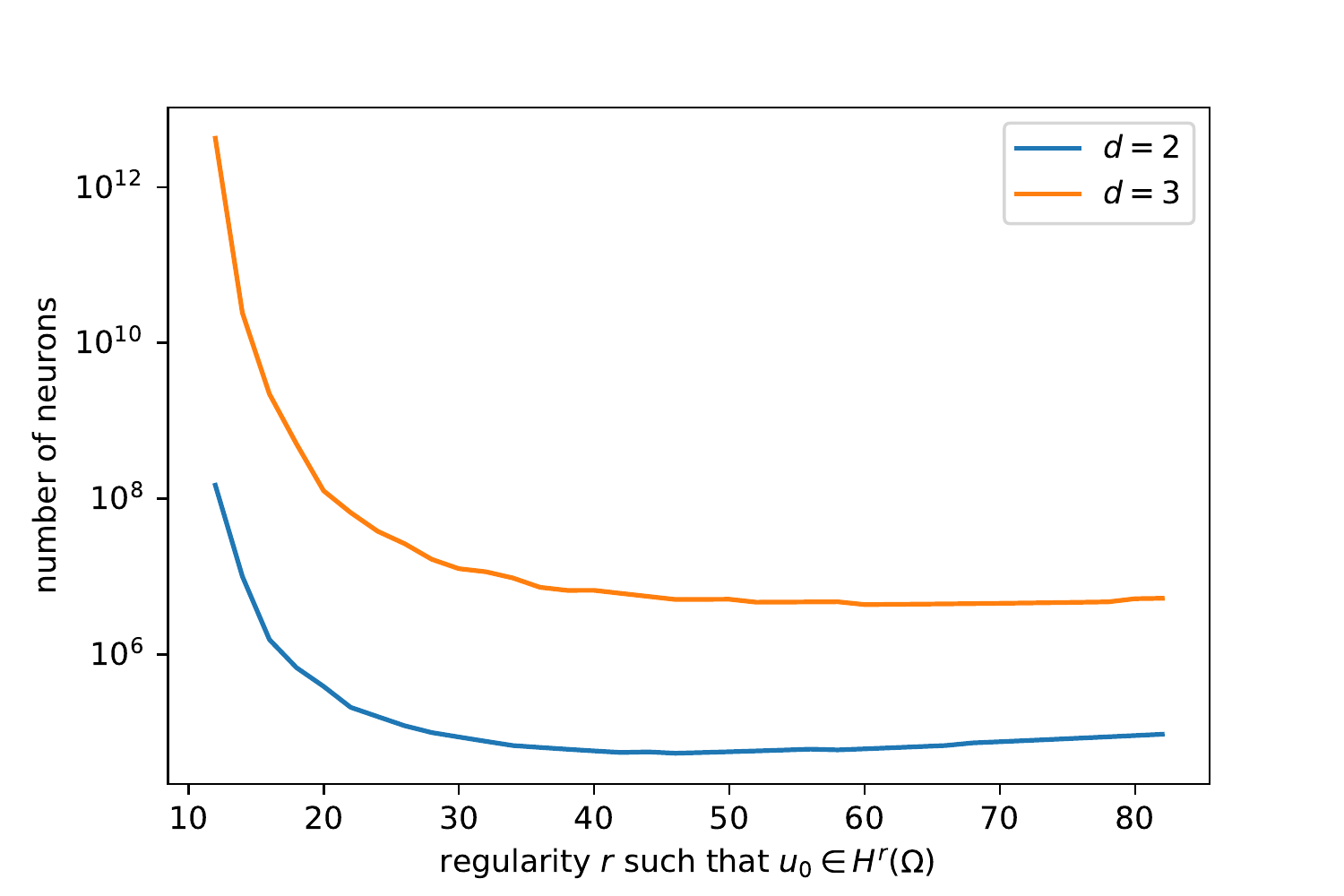}
    \caption{Needed neural network size according to Theorem \ref{thm:pinn-approx} such that \eqref{eq:bound-pinn-res-1}, \eqref{eq:bound-pinn-res-2} and \eqref{eq:bound-pinn-res-3} are all smaller than $1\%$ for varying regularity $r$ of the initial condition i.e., $u_0\in H^r(\mathbb{T}^d)$. }
    \label{fig:nn-size-pinn}
\end{figure}

\subsection{Bound on the total error.} Next, we will show that neural networks for which the (X)PINN residuals are small, will provide a good $L^2$-approximation of the true solution $u:\Omega=D\times [0,T]\to \mathbb{R}^d$, $p:\Omega\to\mathbb{R}$ of the Navier-Stokes equation \eqref{eq:navier-stokes} on the torus $D=\mathbb{T}^d=[0,1)^d$ with periodic boundary conditions. Our analysis can be readily extended to other boundary conditions, such as no-slip boundary condition i.e., $u(x,t)=0$ for all $(x,t)\in \partial D\times [0,T]$, and no-penetration boundary conditions i.e., $u(x,t)\cdot \hn_D=0$ for all $(x,t)\in \partial D\times [0,T]$. 

For neural networks $(u_\theta, p_\theta)$, we define the following PINN-related residuals, 
\begin{align}\label{eq:new-pinn-residuals}
    \begin{split}
        &\rpde = \partial_t u_\theta + (u_\theta\cdot\nabla)u_\theta + \nabla p_\theta - \nu\Delta u_\theta, \qquad
        \rdiv = \div{u_\theta},\\ &\rsu(x) = u_\theta(x)-u_\theta(x+1), \qquad \rsp(x) = p_\theta(x)-p_\theta(x+1),\\ &\rsgu(x) =\nabla u_\theta(x)-\nabla u_\theta(x+1), \qquad \rs = (\rsu, \rsp, \rsgu), \\&\rt = u_\theta(t=0)-u(t=0), 
    \end{split}
\end{align}
where we drop the $\theta$-dependence in the definition of the residuals for notational convenience. 

We will also extend our analysis to the XPINN framework for two subdomains (the extension to more subdomains is straightforward). For this reason, we assume that $D = D_a\cup D_b$, where $D_a$ and $D_b$ are closed with non-overlapping interior $\mathring{D_a}\cap\mathring{D_b}=\emptyset$ and common boundary $\Gamma = D_a\cap D_b$, which we assume to be suitably smooth. We define $\hn_\Gamma$ to point outwards of $D_a$. Figure \ref{fig:xpinn-ab} provides a visualization of this set-up. 

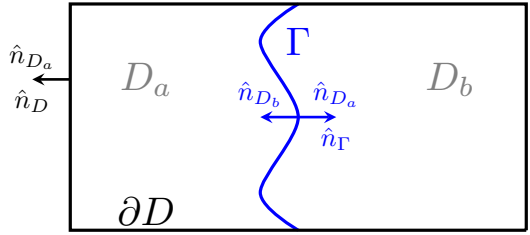
\begin{figure}
    \centering
    \begin{tikzpicture}
    \draw [blue, very thick] plot [smooth ] coordinates {(0,0) (-0.5,0.5)(0,1.5)(-0.5,2.5)(0,3)}; 
    \node[blue, scale=1.5] at (0, 2.5) {$\Gamma$};
    \draw [black, very thick] plot coordinates {(3,0) (-3,0)(-3,3) (3,3) (3,0)};
    \node[gray,scale=1.5] at (-2,2) {$D_a$};
    \node[gray,scale=1.5] at (2,2) {$D_b$};
    \node[black,scale=1.5] at (-2,0.25) {$\partial D$};
    \draw[line width=1pt,blue,-stealth](0,1.5)--(0.5,1.5) node[anchor= south]{$\hn_{D_a}$} node[anchor=north]{$\hn_\Gamma$};
    \draw[line width=1pt,blue,-stealth](0,1.5)--(-0.5,1.5) node[anchor= south]{$\hn_{D_b}$};
    \draw[line width=1pt,black,-stealth](-3,2)--(-3.5,2) node[anchor= south]{$\hn_{D_a}$} node[anchor= north]{$\hn_{D}$};
    \end{tikzpicture}
    \caption{Visualization of the set-up for the XPINN framework with two subdomains. }
    \label{fig:xpinn-ab}
\end{figure}
Following \eqref{eq:xpinn}, the XPINN solution is then defined as 
\begin{equation}\label{eq:xpinn-2subdomains}
    u_\theta = \begin{cases}
      u_\theta^a &\text{in } D_a\setminus \Gamma, \\
      u_\theta^b &\text{in } D_b\setminus \Gamma, \\
      \frac{1}{2}(u_\theta^a +u_\theta^b) &\text{in } \Gamma, 
    \end{cases} \qquad 
     p_\theta = \begin{cases}
      p_\theta^a &\text{in } D_a\setminus \Gamma, \\
      p_\theta^b &\text{in } D_b\setminus \Gamma, \\
      \frac{1}{2}(p_\theta^a +p_\theta^b) &\text{in } \Gamma, 
    \end{cases}
\end{equation}
where $u_\theta^a, u_\theta^b, p_\theta^a, p_\theta^b$ are neural networks. 
In addition to the PINN-related residuals, the following XPINN-related residuals need to be defined, 
\begin{align}\label{eq:new-xpinn-residuals}
    \begin{split}
        \ru &= \max_{j} \abs{(u_\theta^a)_j-(u_\theta^b)_j}, \qquad {\color{black}\rgraduab = \max_{i,j} \abs{\partial_i (u_\theta^a)_j-\partial_i(u_\theta^b)_j}},\\ \rp &= \max_{i,j} \abs{p
        _\theta^a-p_\theta^b}.
    \end{split}
\end{align}
The following theorem then bounds the $L^2$-error of the (X)PINN in terms of the residuals defined above, see also \citep{MM1,UTB1} for versions of the stability argument used below. We write $\abs{\partial D}$, $\abs{\Gamma}$ for the $(d-1)$-dimensional Lebesgue measure of $\partial D$ and $\Gamma$, respectively, and $\abs{D}$ for the $d$-dimensional Lebesgue measure of $D$. 

\begin{theorem}\label{thm:stability}
Let $d\in\mathbb{N}$, $D=\mathbb{T}^d$ and $u\in C^1(D\times [0,T])$ be the classical solution of the Navier-Stokes equation \eqref{eq:navier-stokes}. Let $(u_\theta,p_\theta)$ be a PINN/XPINN with parameters $\theta$, then the resulting $L^2$-error is bounded as follows, 
\begin{align}
\begin{split}
   \int_\Omega \norm{u(x,t)-u_\theta(x,t)}_2^2 dxdt &\leq  \mathcal{C} T \exp(T(2d^2\norm{\nabla u}_{L^\infty(\Omega)}+1)),
\end{split}
\end{align}
where the constant $\mathcal{C}$ is defined as,
\begin{align}
    \begin{split}
     \mathcal{C} = &\: \norm{\rt}^2_{L^2(D)} + \norm{\rpde}^2_{L^2(\Omega)} + C_1\sqrt{T}\bigg[\sqrt{\abs{D}}\norm{\rdiv}_{L^2(\Omega)} + (1+\nu)\sqrt{\abs{\partial D}}\norm{\rs}_{L^2(\partial D\times [0,T])} \\&+ \sqrt{\abs{\Gamma}}\bigg(
     (1+\nu)\norm{\ru}_{L^2(\Gamma\times [0,T])}+\nu{\color{black}\norm{\rgraduab}_{L^2(\Gamma\times [0,T])}}+\norm{\rp}_{L^2(\Gamma\times [0,T])}\bigg)\bigg], 
    \end{split}
\end{align}
and $C_1 = C_1\big(\norm{u}_{C^1},\norm{\hu}_{C^1},\norm{p}_{C^0},\norm{\hp}_{C^0}\big)<\infty$. For PINNs, it holds that $\ru=\rgraduab=\rp=0$. 
\end{theorem}
\begin{proof}
Let $\hu=u_{\theta}-u$ and $\hp = p_\theta -p$ denote the difference between the solution of the Navier-Stokes equations and a PINN with parameter vector $\theta$. Using the Navier-Stokes equations \eqref{eq:navier-stokes} and the definitions of the different residuals, we find after a straightforward calculation that,
\begin{align}
    \begin{split}
        &\rpde = \hu_t + (\hu\cdot\nabla)\hu + (u\cdot\nabla)\hu + (\hu\cdot\nabla)u + \nabla\hp - \nu\Delta\hu,\\
        &\rdiv = \div{\hu}, \qquad \rs(x) = u_\theta(x)-u_\theta(x+1), \qquad \rt = \hu(t=0),\\
        &\ru = \max_{j} \abs{(u_\theta^a)_j-(u_\theta^b)_j}, \quad \rgraduab = \max_{i,j} \abs{\partial_i (u_\theta^a)_j-\partial_i(u_\theta^b)_j}, \quad \rp = \max_{i,j} \abs{p
        _\theta^a-p_\theta^b}.
    \end{split}
\end{align}
Next, we recall the following vector equalities,
\begin{equation}
    \hu\cdot \hu_t = \frac{1}{2}\partial_t\norm{u}_2^2, \quad 
    \hu\cdot((\hu\cdot\nabla)\hu) =\frac{1}{2} (\hu\cdot\nabla)\norm{\hu}_2^2, \quad
    \hu\cdot((u\cdot\nabla)\hu) =\frac{1}{2}(u\cdot\nabla)\norm{\hu}_2^2. 
\end{equation}
We take the inner product of the first equation in \eqref{eq:navier-stokes} and $\hu$, and use the previous vector inequalities to obtain,
\begin{equation}\label{eq:stab1}
    \frac{1}{2}\partial_t\norm{\hu}_2^2 + \frac{1}{2} (\hu\cdot\nabla)\norm{\hu}_2^2 + \frac{1}{2}(u\cdot\nabla)\norm{\hu}_2^2 + \hu\cdot((\hu\cdot\nabla)u) + (\hu\cdot\nabla)\hp -\nu\hu\cdot\Delta\hu= \hu\cdot \rpde
\end{equation}
Now let $\Lambda\subset D$ be such that $\partial \Lambda$ is piecewise smooth with outward normal vector $\hn_\Lambda$. Denote by $T_\Lambda$ the corresponding trace operator. Integrating \eqref{eq:stab1} over $\Lambda$ and integrating by parts yields, 
\begin{align}\label{eq:stab2}
    \begin{split}
        \frac{d}{dt}\int_\Lambda \norm{\hu}_2^2 dx =& \int_\Lambda \rdiv (\norm{\hu}_2^2+2\hp)dx - \int_{\partial \Lambda} T_\Lambda(\hu)\cdot \hn_\Lambda(\norm{\hu}_2^2+2\hp)ds(x) \\
        &-2\int_\Lambda \hu\cdot((\hu\cdot\nabla)u)dx -2 \nu\sum_{j=1}^d \int_\Lambda \norm{\nabla \hu_j }_2^2 dx \\
        &+ 2\nu\sum_{j=1}^d\int_{\partial\Lambda} T_\Lambda(\hu_j) (\hn_\Lambda\cdot T_\Lambda(\nabla \hu_j)) ds(x) + 2\int_\Lambda \hu\cdot \rpde \:dx.
    \end{split}
\end{align}
The use of the trace operator $T_\Lambda$ is necessary since the trace of $\hu$ on $\partial \Lambda$ might not agree with the actual definition of $\hu$ as in  \eqref{eq:xpinn-2subdomains}.
We then find
\begin{align}\label{eq:stab3}
    \begin{split}
&\sum_{i=a}^b \int_{\partial D_i} T_{D_i}(\hu^i)\cdot \hn_{D_i}(\norm{\hu^i}_2^2+2\hp^i)ds(x) - \int_{\partial D} \hu\cdot  \hn_D(\norm{\hu}_2^2+2\hp)ds(x)\\
&=  \int_{\Gamma} \hu^a\cdot  \hn_\Gamma(\norm{\hu^a}_2^2+2\hp^a)ds(x) - \int_{\Gamma} \hu^b\cdot  \hn_\Gamma(\norm{\hu^b}_2^2+2\hp^b)ds(x)\\
&= \int_{\Gamma} (u_\theta^a-u_\theta^b)\cdot  \hn_\Gamma(\norm{\hu^a}_2^2+2\hp^a)ds(x) + \int_{\Gamma} \hu^b\cdot  \hn_\Gamma(\norm{\hu^a}_2^2-\norm{\hu^b}_2^2+2(p_\theta^a-p_\theta^b))ds(x)
    \end{split}
\end{align}
And similarly, 
\begin{align}\label{eq:stab4}
    \begin{split}
&\sum_{i=a}^b \int_{\partial D_i} T_{D_i}(\hu_j^i) (\hn_{D_i}\cdot T_{D_i}(\nabla \hu_j^i)) ds(x) - \int_{\partial D} \hu_j (\hn_D\cdot \nabla \hu_j) ds(x)\\
&= \int_{\Gamma} \hu_j^a (\hn_\Lambda\cdot \nabla \hu_j^a) ds(x) - \int_{\Gamma} \hu_j^b (\hn_\Lambda\cdot \nabla \hu_j ^b) ds(x)\\
&= \int_{\Gamma} ((u_\theta^a)_j-(u_\theta^b)_j) (\hn_\Lambda\cdot \nabla \hu_j^a) ds(x) - \int_{\Gamma} \hu_j^b (\hn_\Lambda\cdot \nabla ((u_\theta^a)_j-(u_\theta^b)_j)) ds(x)
\end{split}
\end{align}
Moreover, we calculate that for a constant $C_1\big(\norm{u}_{C^1},\norm{\hu}_{C^1},\norm{p}_{C^0},\norm{\hp}_{C^0}\big)$ it holds that,
\begin{align}\label{eq:stab5}
\begin{split}
    &-\int_D \hu\cdot((\hu\cdot\nabla)u)dx \leq d^2\norm{\nabla u}_{L^\infty(\Omega)} \int_D \norm{\hu}^2_2dx,\\
    & \abs{\int_{\partial D} \hu\cdot  \hn_D(\norm{\hu}_2^2+2\hp)ds(x)} \leq C_1 \left(\norm{\rsu}_{L^1(\partial D)}+\norm{\rsp}_{L^1(\partial D)}\right),\\
    & \int_{\partial D} \hu_j (\hn_D\cdot \nabla \hu_j) ds(x) \leq C_1 \left(\norm{\rsu}_{L^1(\partial D)}+\norm{\rsgu}_{L^1(\partial D)}\right),\\
\end{split}
\end{align}
where $\Omega = D\times [0,T]$. Now, summing \eqref{eq:stab1} over the different $\Lambda = D_i$, integrating over the interval $[0,\tau]\subset [0,T]$ and using \eqref{eq:stab2}, \eqref{eq:stab3}, \eqref{eq:stab4} we find that,
\begin{align}\label{eq:stab6}
\begin{split}
    \int_D \norm{\hu(x,\tau)}_2^2 dx  \leq&\: \norm{\rt}^2_{L^2(D)} + C_1\sqrt{T\abs{D}}\norm{\rdiv}_{L^2(\Omega)} + C_1 (1+\nu)\sqrt{T\abs{\partial D}}\norm{\rs}_{L^2(\partial D\times [0,T])} \\&+ C_1(1+\nu)\sqrt{T\abs{\Gamma}}\max_{j}\norm{(u_\theta^a)_j-(u_\theta^b)_j}_{L^2(\Gamma\times [0,T])} \\ &+ C_1\sqrt{T\abs{\Gamma}} \norm{p_\theta^a-p_\theta^b}_{L^2(\Gamma\times [0,T])} +  2d^2\norm{\nabla u}_{L^\infty(\Omega)} \int_{D\times [0,\tau]} \norm{\hu(x,t)}_2^2 dxdt \\
    &+  C_1\nu \sqrt{T\abs{\Gamma}} \max_{i,j}\norm{\partial_i (u_\theta^a)_j-\partial_i(u_\theta^b)_j}_{L^2(\Gamma\times [0,T])} \\&+ \norm{\rpde}^2_{L^2(\Omega)} + \int_{D\times [0,\tau]} \norm{\hu(x,t)}_2^2 dxdt, 
\end{split}
\end{align}
where $\rs = (\rsu, \rsp, \rsgu)$ as in \eqref{eq:new-pinn-residuals}. 
Using Grönwall's inequality and integrating over $[0,T]$, we find that,
\begin{align}
\begin{split}
   \int_\Omega \norm{\hu(x,t)}_2^2 dxdt &\leq  \mathcal{C} T \exp(T(2d^2\norm{\nabla u}_{L^\infty(\Omega)}+1)),
\end{split}
\end{align}
where the constant $\mathcal{C}$ is defined as,
\begin{align}
    \begin{split}
     \mathcal{C} = &\: \norm{\rt}^2_{L^2(D)} + \norm{\rpde}^2_{L^2(\Omega)} + C_1\sqrt{T}\bigg[\sqrt{\abs{D}}\norm{\rdiv}_{L^2(\Omega)} + (1+\nu)\sqrt{\abs{\partial D}}\norm{\rs}_{L^2(\partial D\times [0,T])} \\&+ \sqrt{\abs{\Gamma}}\bigg(
     (1+\nu)\norm{\ru}_{L^2(\Gamma\times [0,T])}+\nu\norm{\rgraduab}_{L^2(\Gamma\times [0,T])}+\norm{\rp}_{L^2(\Gamma\times [0,T])}\bigg)\bigg] .
    \end{split}
\end{align}
\end{proof}

\begin{remark}
Although the existence of a $C^1$ solution of the Navier-Stokes solution is guaranteed by Theorem \ref{thm:NS-existence}, it is still possible that $\norm{\nabla u}_{L^\infty(\Omega)}$ becomes very large, e.g. for complicated solutions characterized by strong vorticity \citep{MM1}. In such a case, Theorem \ref{thm:stability} indicates that the generalization error might be large. 
\end{remark}

\begin{remark}
For PINNs, the $L^2$-error is bounded uniquely in terms of residuals that are a part of the PINN generalization error \eqref{eq:generalization-error-pinn}. This implies that for neural networks with a small PINN loss the corresponding $L^2$-error will be small as well, provided that the $C^1$-norm of the network does not blow up. This affirmatively answers question Q2. For XPINNs, we can see that the XPINN-specific residuals $\ru$ and $\rp$ (as defined in \eqref{eq:new-xpinn-residuals}) are equivalent with the $\ru$ residual in the XPINN generalization error \eqref{eq:generalization-error-xpinn}. 
{\color{black}The residual $\rgraduab$ however does not show up in the original XPINN framework, and should therefore be added to the XPINN loss function \eqref{eq:generalization-error-xpinn} to theoretically guarantee a small $L^2$-error. }
\end{remark}

\begin{remark}
Variants of Theorem \ref{thm:stability} for different kinds of boundary conditions can be proven in the same way as above. For example, the statement from Theorem \ref{thm:stability} still holds for no-slip boundary conditions i.e., $u(x,t)=0$ for all $(x,t)\in \partial D\times [0,T]$, if one defined the spatial boundary residual as $\rs = u_\theta$. 
\end{remark}

\begin{remark}
Although the focus in this paper lies on solving the Navier-Stokes equations for the velocity, we want to note that one can also prove a stability result for $\norm{p-p_\theta}_{L^2(\Omega)}$ in a similar spirit to Theorem \ref{thm:stability}. The main steps consists of taking the divergence of the Navier-Stokes equations, using the identity \eqref{eq:poisson-p} and rewriting the result in terms of the different residuals. 
\end{remark}

The existence of a PINN (XPINN) with an arbitrarily small $L^2$-error is a simple byproduct of the proof of Theorem \ref{thm:pinn-approx}. For completeness, we show that one can also use Theorem \ref{thm:stability} to obtain a quantitative convergence result on the $L^2$-error of the PINN approximation of the solution of the Navier-Stokes equation in terms of the number of neurons of the neural network.

\begin{corollary}\label{cor:L2-pinn}
Let $n\geq 2$, $d,r,k\in\mathbb{N}$, where $k\geq 3$
and let $u_0\in H^r(\mathbb{T}^d)$ with $r>\frac{d}{2}+2k$ and $\div{u_0}=0$. It holds that: 
\begin{itemize}
    \item there exist $T>0$ and a classical solution $u$ to the Navier-Stokes equations such that $u\in H^{k}(\Omega)$, $\nabla p\in H^{k-1}(\Omega)$, $\Omega = \mathbb{T}^d\times [0,T]$, and $u(t=0)=u_0$,  
    \item there exist constants $C, \beta>0$ such that for every $N\in \mathbb{N}$, there exist tanh neural networks $\hu_j$, $1\leq j\leq d$, and $\widehat{p}$, each with two hidden layers, of widths $3\left\lceil\frac{k+n-2}{2}\right\rceil\binom{d+k-1}{d}+\lceil TN\rceil +dN$ and $3\left\lceil\frac{d+n}{2}\right\rceil \binom{2d+1}{d}\lceil TN \rceil N^{d}$, such that for every $1\leq j\leq d$, 
    \begin{equation}
        \norm{u-\hu}_{L^2(\Omega)} \leq C \ln^\kappa(\beta N) N^{\frac{-k+1}{2}}.
    \end{equation}
    The value of $C>0$ follows from the proof, $\beta>0$ and the network weight growth are as in Theorem \ref{thm:pinn-approx}, and $\kappa=2$ for $k=3$ and $\kappa=\frac{1}{2}$ for $k\geq 4$.
\end{itemize}
\end{corollary}

\begin{proof}
The corollary is a direct consequence of Theorem \ref{thm:stability} and Theorem \ref{thm:pinn-approx} and its proof.
\end{proof}

\subsection{Bounds on the total error in terms of training error.}\label{sec:33}
Next, we answer the question Q3, raised in the introduction, by providing a bound of the generalization error in terms of the training error and the size of the training set $\S$, where $u_{\theta^*(\S)}$ is the PINN that minimizes the training loss. Combined with Theorem \ref{thm:stability}, it will enable us to bound the total error (the $L^2$-mismatch between the exact solution of \eqref{eq:navier-stokes} and the trained PINN) in terms of the training error and size of the training set. 

As already announced in Section \ref{sec:quad}, we will focus on training sets obtained using the \emph{midpoint rule} $\qu{M}$ for simplicity. 
For $f\in\{\rpde^2, \rdiv^2\}$ and $\Lambda = \Omega = D\times [0,T]$ we obtain the quadrature $\qu{M}^\text{int}$, for $f=\rt^2$ and $\Lambda = D$ we obtain the quadrature $\qu{M}^t$ and for $f=\rs^2$ and $\Lambda = \partial D\times [0,T]$ we obtain the quadrature $\qu{M}^{s}$. For XPINNs, one additionally needs to consider the quadrature $\qu{M}^\Gamma$ obtained for $f\in\{\ru^2, \rgraduab^2, \rp^2\}$ and $\Lambda = \Gamma $. 

This notation allows us to write the PINN loss \eqref{eq:training-loss-pinn} in a compact manner,
\begin{equation}
    \label{eq:etrain}
\begin{aligned}
    \Et(\theta,\S)^2 &= \Et^\pde(\theta,\S_\inte)^2 +  \Et^\divv(\theta,\S_\inte)^2 +\Et^s(\S_s)^2+ \Et^t(\theta,\S_t)^2,\\
    &= \qu{M_{\text{int}}}^\text{int}[\rpde^2] + \qu{M_{\text{int}}}^\text{int}[\rdiv^2]+  \qu{M_{s}}^s[\rs^2] +  \qu{M_{t}}^t[\rt^2],
\end{aligned}
\end{equation}
Using this notation and Theorem \ref{thm:stability} from the previous section, we obtain the following theorem that bounds the $L^2$-error of a neural network in terms of the training loss and the number of training points. In particular, it applies to the trained PINN $u_{\theta^*(\S)}$. 

\begin{theorem}\label{thm:generalization}
Let $T>0$, $d\in\mathbb{N}$, let $(u,p)\in C^4(\mathbb{T}^d\times [0,T])$ be the classical solution of the Navier-Stokes equation \eqref{eq:navier-stokes} and let $(u_\theta,p_\theta)$ be a PINN with parameters $\theta\in\Theta_{L,W,R}$ (cf. Definition \ref{def:nn}). Then the following error bound holds,
\begin{equation}
    \label{eq:etrain1}
\begin{aligned}
   \int_\Omega \norm{u(x,t)-u_\theta(x,t)}_2^2 dxdt &\leq  \mathcal{C}(M) T \exp(T(2d^2\norm{\nabla u}_{L^\infty(\Omega)}+1))\\
   &=  \bigO\left(\Et(\theta,\S)^2 + M_t^{-\frac{2}{d}} + M_\mathrm{int}^{-\frac{1}{d+1}} + M_s^{-\frac{1}{d}}\right).
\end{aligned}
\end{equation}
In the above formula, the constant $\mathcal{C}(M)$ is defined as,
\begin{align}
    \begin{split}
     \mathcal{C}(M) = &\: \Et^t(\theta,\S_t)^2 + C_t M_t^{-\frac{2}{d}} + \Et^\pde(\theta,\S_\inte)^2 +C_\pde M_\mathrm{int}^{-\frac{2}{d+1}}\\&+ C_1{T}^{\frac{1}{2}}\bigg[\Et^\divv(\theta,\S_\inte) +C_\divv M_\mathrm{int}^{-\frac{1}{d+1}}+(1+\nu) \big( \Et^s(\theta,\S_s) + C_s M_s^{-\frac{1}{d}}\big)\bigg],\\
    \end{split}
\end{align}
and where,
\begin{align}\label{eq:bounds-constants}
    \begin{split}
        C_1 &\lesssim \norm{u}_{C^1} + \norm{p}_{C^0} + \norm{\hu}_{C^1} + \norm{\hp}_{C^0}\\
        &\lesssim \norm{u}_{C^1} + \norm{p}_{C^0} + (d+1)^2 \left(16e^2 W^3R \norm{\sigma}_{C^1}\right)^{L},\\
        C_t &\lesssim \norm{u}_{C^2}^2 + \norm{\hu}_{C^2}^2 \lesssim \norm{u}_{C^2}^2 + \left(e^2 2^6 W^3R^2 \norm{\sigma}_{C^2}\right)^{2L},\\
        C_\pde &\lesssim  \norm{\hu_j}_{C^4}^2 \lesssim\left(2e^2 4^4 W^3R^4 \norm{\sigma}_{C^4}\right)^{4L},\\
        C_\divv, C_s &\lesssim \norm{\hu_j}_{C^3} \lesssim\left(4e^2 3^4 W^3R^3 \norm{\sigma}_{C^3}\right)^{3L/2}.
    \end{split}
\end{align}
\end{theorem}

\begin{proof}
The main error estimate of the theorem follows directly from combining Theorem \ref{thm:stability} with the quadrature error formula \eqref{eq:quad-error}. The complexity of $C_1$ follows from Theorem \ref{thm:stability} and Lemma \ref{lem:nn-cn-bound}, which states that
\begin{equation}
    \norm{\hu_j}_{C^n} \leq 16^L(d+1)^{2n} \left(e^2 n^4 W^3R^n \norm{\sigma}_{C^n}\right)^{nL}
\end{equation}
for $n\in\mathbb{N}$, all $j$ and similarly for $\hp$. The complexities of the other constants then follow from this formula and the observation that for every residual $\mathcal{R}_q$ it holds that $ \norm{\mathcal{R}_q^2}_{C^n} \leq 2^n \norm{\mathcal{R}_q}_{C^n}^2$ (from the general Leibniz rule). For instance, we obtain in this way that
\begin{equation}
    C_t \lesssim \norm{\mathcal{R}_t^2}_{C^2} \lesssim  \norm{u}_{C^2}^2 + \norm{\hu}_{C^2}^2 \lesssim \norm{u}_{C^2}^2 + \left(e^2 2^6 W^3R^2 \norm{\sigma}_{C^2}\right)^{2L}. 
\end{equation}

In a similar way, one can calculate that
\begin{equation}
    C_\pde \lesssim \norm{\mathcal{R}_\pde^2}_{C^2} \lesssim \norm{\hu}_{C^4}^2 \lesssim  \left(2e^2 4^4 W^3R^4 \norm{\sigma}_{C^4}\right)^{4L}.
\end{equation}
Finally, we find that 
\begin{equation}
    C_s^2, C_\divv^2 \lesssim \norm{\mathcal{R}_s^2}_{C^2}, \norm{\mathcal{R}_\divv^2}_{C^2} \lesssim \norm{\hu}_{C^3}^2 \lesssim \left(4e^2 3^4 W^3R^3 \norm{\sigma}_{C^3}\right)^{3L}. 
\end{equation}
\end{proof}

\begin{remark}
For XPINNs, an entirely analogous result can be proven using the same approach. 
\end{remark}

\begin{remark}
The upper bounds on the constants in \eqref{eq:bounds-constants} depend polynomially on the network width $W$ but exponentially on the network depth $L$. These bounds seem to suggest that one might expect a smaller $L^2$-error for a rather shallow (but wide) network than for a very deep network. In Theorem \ref{thm:pinn-approx}, we have already proven explicit error bounds for a neural network with only two hidden layers. It is important to note that the constants in \eqref{eq:bounds-constants} can be estimated from the network used in practice, and that there is no need to use the (over)estimates from Theorem \ref{thm:pinn-approx} in this case. 
\end{remark}

\begin{remark}
If we assume that the optimization algorithm used to minimize the training loss i.e., to solve \eqref{eq:opt}, finds a global minimum, then one can prove that the training error in Theorem \ref{thm:generalization} is small if the training set and hypothesis space is large enough.  
To see this, first fix an error tolerance $\epsilon$ and observe that for the network $\hu$ that was constructed in Theorem \ref{thm:pinn-approx} it holds that all relevant PINN residuals and therefore also the generalization error $\Eg(\theta_{\hu})$ \eqref{eq:generalization-error-pinn} are of order $\bigO(\varepsilon)$.
If one then constructs the training set such that $M_\mathrm{int} \sim \epsilon^{-\frac{d+1}{2}}$ and $M_t \sim M_s \sim \epsilon^{-\frac{d}{2}}$ then it holds that $\Et^q(\theta_\Psi) \leq \epsilon + \Eg^q(\theta_\Psi)$ for $q \in\{s,t,\mathrm{div}, \mathrm{PDE}\}$ and as a consequence that $ \Et(\theta_{\hu},\S) = \bigO(\epsilon)$. 
If the optimization algorithm reaches a global minimum, the training loss of $u_{\theta^*(\S)}$ will be upper bounded by that of $\hu$. Therefore it also holds that $ \Et(\S) = \bigO(\epsilon)$. 
\end{remark}

{\color{black}

With these remarks in place, we now discuss how Theorem \ref{thm:generalization} answers question Q3, raised in the introduction, which asked whether the generalization error $\Eg(\theta^*)$ will be small if the training error is small and the training set is sufficiently large. 

First of all, Theorem \ref{thm:generalization} can be used as an \emph{a posteriori} error estimate for the trained PINN as \eqref{eq:etrain1} shows that the generalization error can be upper bounded in terms of the various PINN-related residuals, as well as the $C^k$-norms of the trained neural network and the training set sizes. The power of  \eqref{eq:etrain1} as an a posteriori error estimate will be demonstrated with numerical experiments in Section \ref{sec:4}. 

Next, we combine Theorem \ref{thm:pinn-approx} with Theorem \ref{thm:generalization} to prove an \emph{a priori} error estimate. The following result states that the total $L^2$-error of the trained PINN $\Vert u-u_{\theta^*(\S)}\Vert_{L^2(D\times [0,T])}$ can be made arbitrarily small if the optimization procedure for \eqref{eq:opt} leads to a global minimum, the training set and the underlying spaces of neural networks are sufficiently large, and the solution $u$ is sufficiently smooth. 

\begin{corollary}
Let $\epsilon>0$, $T>0$, $d\in\mathbb{N}$, $k>6(3d+8) =: \gamma$, let $(u,p)\in H^k(\mathbb{T}^d\times [0,T])$ be the classical solution of the Navier-Stokes equation \eqref{eq:navier-stokes}, let the hypothesis space $\Theta$ satisfy $R \geq \epsilon^{-1/(k-\gamma)}\ln(1/\epsilon) $, $W \geq \epsilon^{-(d+1)/(k-\gamma)}$ and $L \geq 3$, and let $(u_{\theta^*(\S)},p_{\theta^*(\S)})$ be the PINN that solves \eqref{eq:opt} where the training set $\S$ satisfies  $M_t \geq \epsilon^{-d(1+\gamma/(k-\gamma))}$, $M_\inte \geq \epsilon^{-2(d+1)(1+\gamma/(k-\gamma))}$ and $M_s \geq \epsilon^{-2d(1+\gamma/(k-\gamma))}$. It holds that
\begin{equation}
    \Vert u-u_{\theta^*(\S)}\Vert_{L^2(D\times [0,T])} = \bigO(\epsilon).
\end{equation}
\end{corollary}
\begin{proof}
For arbitrary $N\in\N$, we set $R=\bigO(N\ln(N))$, $W = \bigO(N^{d+1})$ and $L=3$ and we let $\Theta := \Theta_{L,W,R}$ as in Definition \ref{def:nn}. From Theorem \ref{thm:pinn-approx} we know that there exists $\hat{\theta}\in\Theta$ for which $\Eg(\hat{\theta}) = \bigO(\ln^2(N)N^{-k+2})$. 

Now let $\theta^*(\S)$ be the parameter that minimizes $\min_{\theta\in\Theta} \Et(\theta,\S)$ as in \eqref{eq:opt}. We note that Theorem \ref{thm:generalization} implies that
\begin{equation}
    \Vert u-u_{\theta^*(\S)}\Vert_{L^2(D\times [0,T])}^2\lesssim (W^3R^4)^{4L}\left(\Et(\theta^*(\S),\S)^2 + M_t^{-\frac{2}{d}} + M_\mathrm{int}^{-\frac{1}{d+1}} + M_s^{-\frac{1}{d}}\right). 
\end{equation}
By definition, it must hold that $\Et(\theta^*(\S),\S) \leq \Et(\hat{\theta},\S)$. We use the same quadrature rules as in Theorem \ref{thm:generalization} to bound $\Et(\hat{\theta},\S)$ by $\Eg(\hat{\theta})$ and rewrite the previous bound as,
\begin{equation}
    \Vert u-u_{\theta^*(\S)}\Vert_{L^2(D\times [0,T])}^2 \lesssim (W^3R^4)^{4L}\left(\Eg(\hat{\theta})^2 + M_t^{-\frac{2}{d}} + M_\mathrm{int}^{-\frac{1}{d+1}} + M_s^{-\frac{1}{d}}\right). 
\end{equation}
Rewriting everything in terms of $N$, we find that
\begin{equation}
\begin{split}
    \Vert u-u_{\theta^*(\S)}\Vert_{L^2(D\times [0,T])}^2 &\lesssim \left(N^{3(d+1)}N^4\ln^4(N)\right)^{12}\left(\ln^4(N)N^{-2k+4} + M_t^{-\frac{2}{d}} + M_\mathrm{int}^{-\frac{1}{d+1}} + M_s^{-\frac{1}{d}}\right) \\
    &\lesssim N^{12(3d+8)}\left(N^{-2k} + M_t^{-\frac{2}{d}} + M_\mathrm{int}^{-\frac{1}{d+1}} + M_s^{-\frac{1}{d}}\right).
\end{split}
\end{equation}
We now choose $N$ and the training set sizes in such a way that the RHS of the above inequality is $\bigO(\epsilon^2)$. Concretely, this means setting $N = \epsilon^{-1/(k-\gamma)}$, $M_t = \epsilon^{-d(1+\gamma/(k-\gamma))}$, $M_\inte = \epsilon^{-2(d+1)(1+\gamma/(k-\gamma))}$ and $M_s = \epsilon^{-2d(1+\gamma/(k-\gamma))}$. This concludes the proof of the corollary.
\end{proof}
}

\section{Numerical experiments}
\label{sec:4}
In this section, we seek to illustrate the bounds on error of the PINN and XPINN approximations of the Navier-Stokes equations \eqref{eq:navier-stokes}, empirically with a numerical experiment.  

To this end, we consider the Navier-Stokes equations in two space dimensions and initial data that corresponds to the Taylor-Green vortex test case, which is an unsteady flow of decaying vortices. The exact closed form solutions of Taylor-Green vortex problem are given by
\begin{align*}
u(t,x,y) &= - \cos(\pi x) \sin(\pi y) \text{exp}(-2\pi^2\nu t)  
\\ v(t,x,y) &= \sin(\pi x) \cos(\pi y) \text{exp}(-2\pi^2\nu t) 
\\ p(t,x,y) &= -\frac{\rho}{4}\left[\cos(2\pi x) + \cos(2\pi y)\right] \text{exp}(-4 \pi^2\nu t)  
\end{align*}
The spatio-temporal domain is $x, y \in [0.5,4.5]^2$ and $t \in [0,1].$

The Taylor-Green vortex serves two key requirements in our context. First, it provides an analytical solution of the Navier-Stokes equations and enables us to evaluate $L^2$-errors with respect to this exact solution and without having to consider further (numerical) approximations. Second, the underlying solution is clearly smooth enough to fit the regularity criteria of all our error estimates, presented in the previous section. 

We will approximate the Taylor-Green vortex with PINNs and XPINNs. In case of XPINNs, we decompose the domain into two subdomains along $x$-axis ($x\geq 2.5$ and $x<2.5$) where separate neural networks are employed. On the common interface we used 300 points for stitching these two subdomains together. The value of density is set at $\rho = 1$. An ensemble training procedure is performed to find the correlation between the total error ($\mathcal{E}$) and the training error ($\mathcal{E}_T$) for different values of $\nu$. For the neural network training we used full batch with Adam optimizer for the first 20000 number of iterations, followed by L-BFGS optimizer \citep{byrd1995limited} for another 60000 iterations or till convergence. The number of layers in both PINN and XPINN are 2 (as suggested by the theory) with 80 neurons in each layer, and the quadrature points are 27K, which are obtained using mid-point rule. The learning rate is 8e-4, and the activation function is hyperbolic tangent in both cases. 
\begin{figure} [htpb] 
\centering
\includegraphics[trim=0cm 0cm 0cm 0cm, clip=true, scale=0.44, angle = 0]{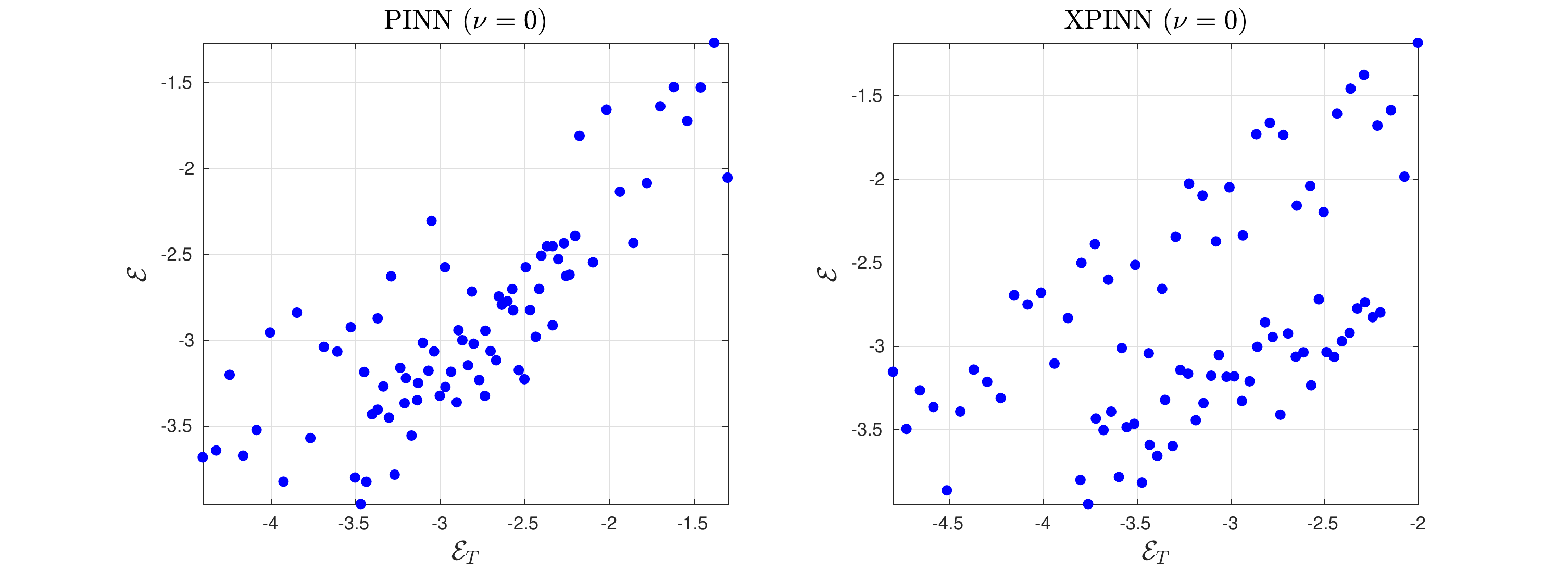}
\includegraphics[trim=0cm 0cm 0cm 0cm, clip=true, scale=0.44, angle = 0]{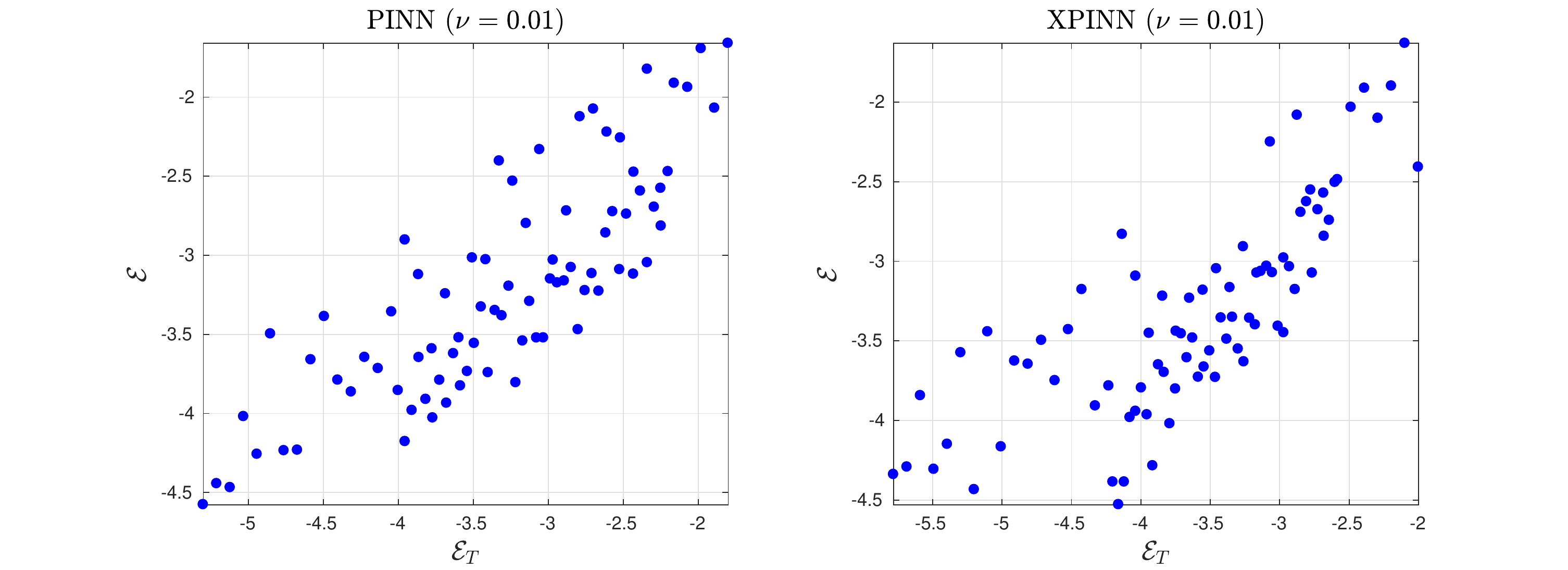}
\includegraphics[trim=0cm 0cm 0cm 0cm, clip=true, scale=0.44, angle = 0]{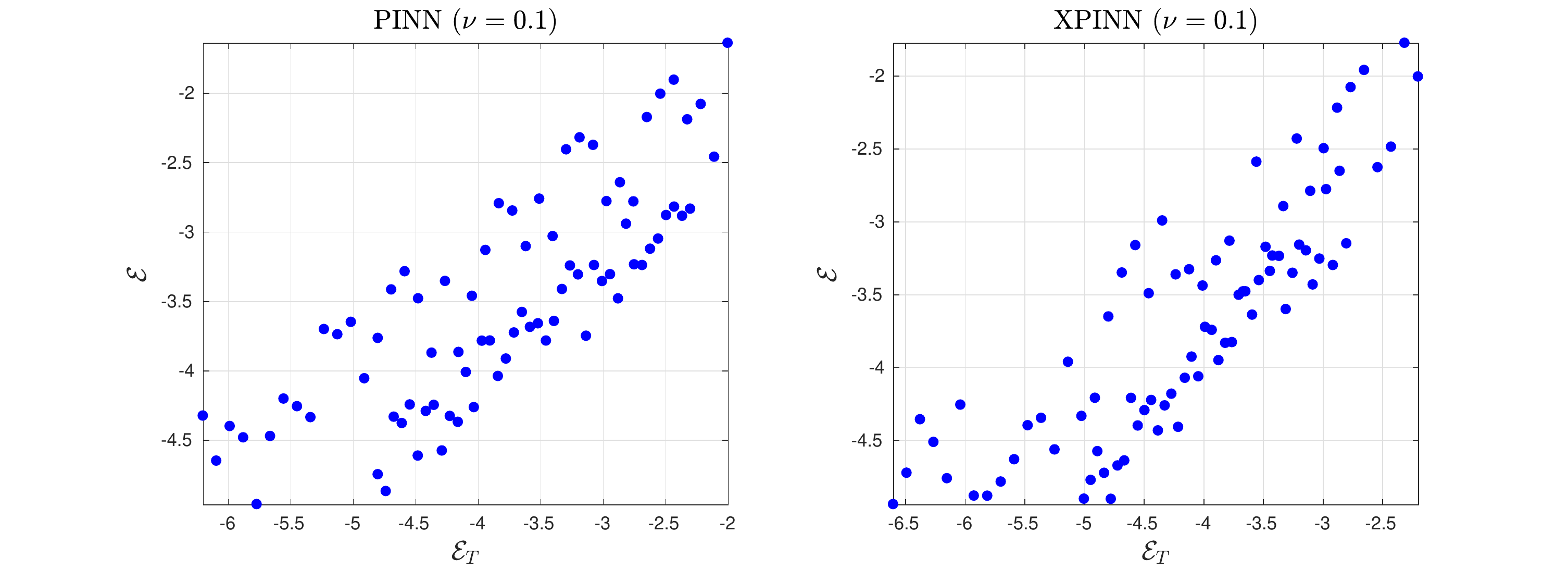}
\caption{Log of training error vs. log of total error for each parameter (weights and biases) configuration during ensemble training. We used $\nu = 0, 0.01$ and 0.1.}
\label{fig:fig1}
\end{figure}
We train the networks 80 times with different set of initialization to weights and biases. Figure \ref{fig:fig1} shows the log of training error vs. log of total error for each parameter configuration during ensemble training with three different values of viscosity $\nu$. As seen from this figure, total error $\Etot = \|u - u^{\ast}\|_{L^2}$ and the training error $\Et$ \eqref{eq:etrain} are very tightly correlated (along the diagonal in Figure \ref{fig:fig1}). In particular and consistent with the estimates in Theorem \ref{thm:generalization}, a small training error implies a small total error. Moreover, we see from Figure \ref{fig:fig1} that the total error $\Etot$ approximately scales as the square root of training error i.e., $\Etot \lesssim \sqrt{\Et}$, which is also consistent with the bounds in Theorem \ref{thm:generalization}.

Next, we investigate the behavior of the total and training errors by varying the number of quadrature points. To this end, we train both PINNs and XPINNs 20 times with different parameter initializations and plot the mean and standard deviation of the errors as shown in Figure \ref{fig:fig2}. All results are of a neural network architecture with 2 hidden layers, with 80 neurons in each layer and the hyperbolic tangent activation function. Moreover, the learning rate is the same as before.

\begin{figure} [htpb] 
\centering
\includegraphics[trim=0cm 0cm 0cm 0cm, clip=true, scale=0.36, angle = 0]{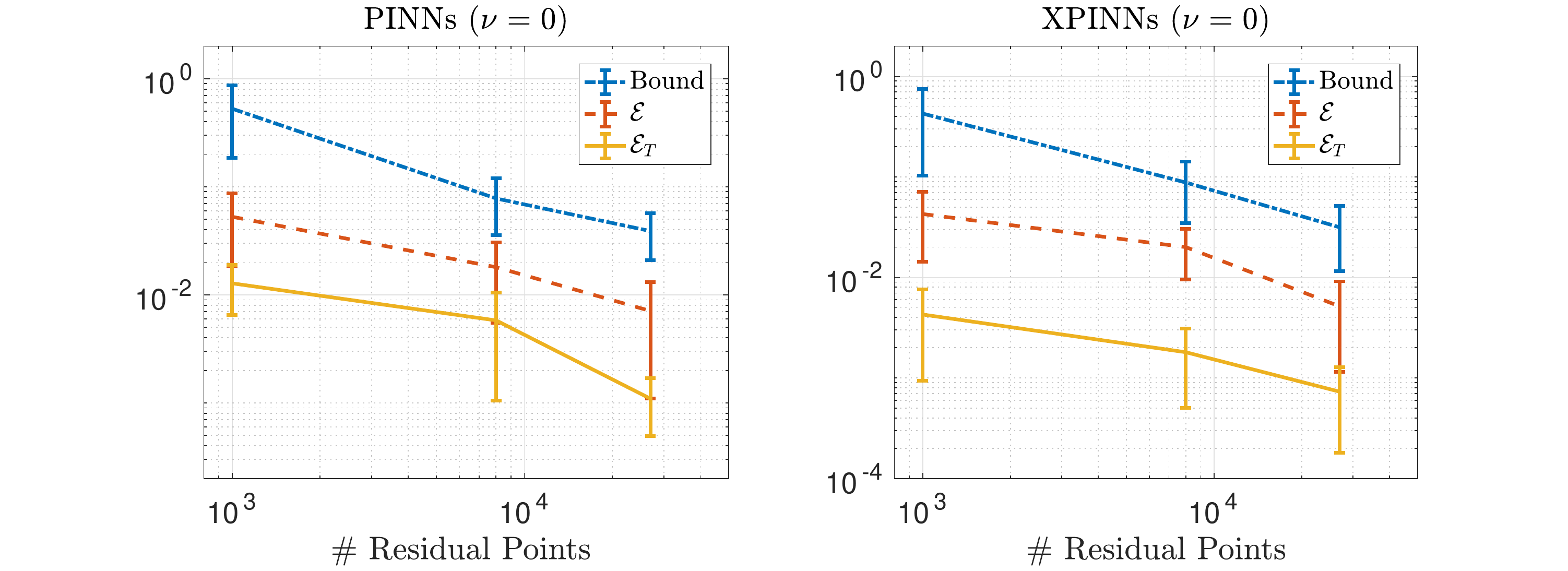}
\includegraphics[trim=0cm 0cm 0cm 0cm, clip=true, scale=0.36, angle = 0]{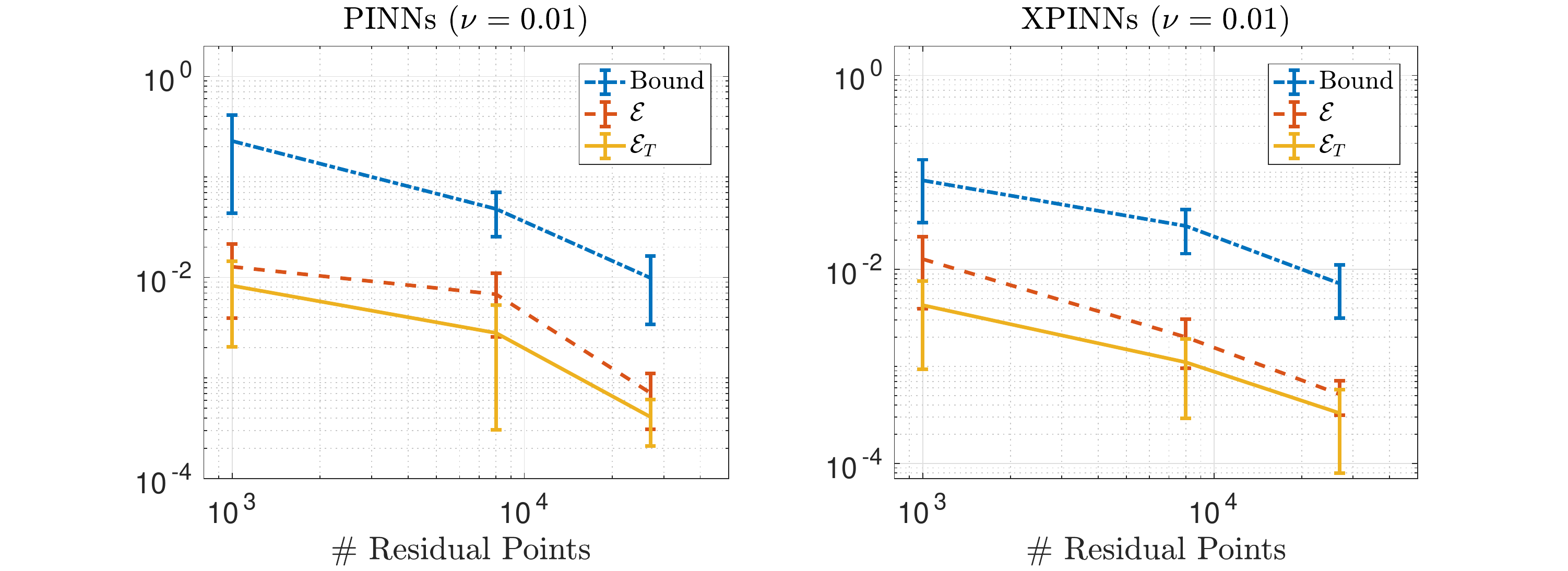}
\includegraphics[trim=0cm 0cm 0cm 0cm, clip=true, scale=0.36, angle = 0]{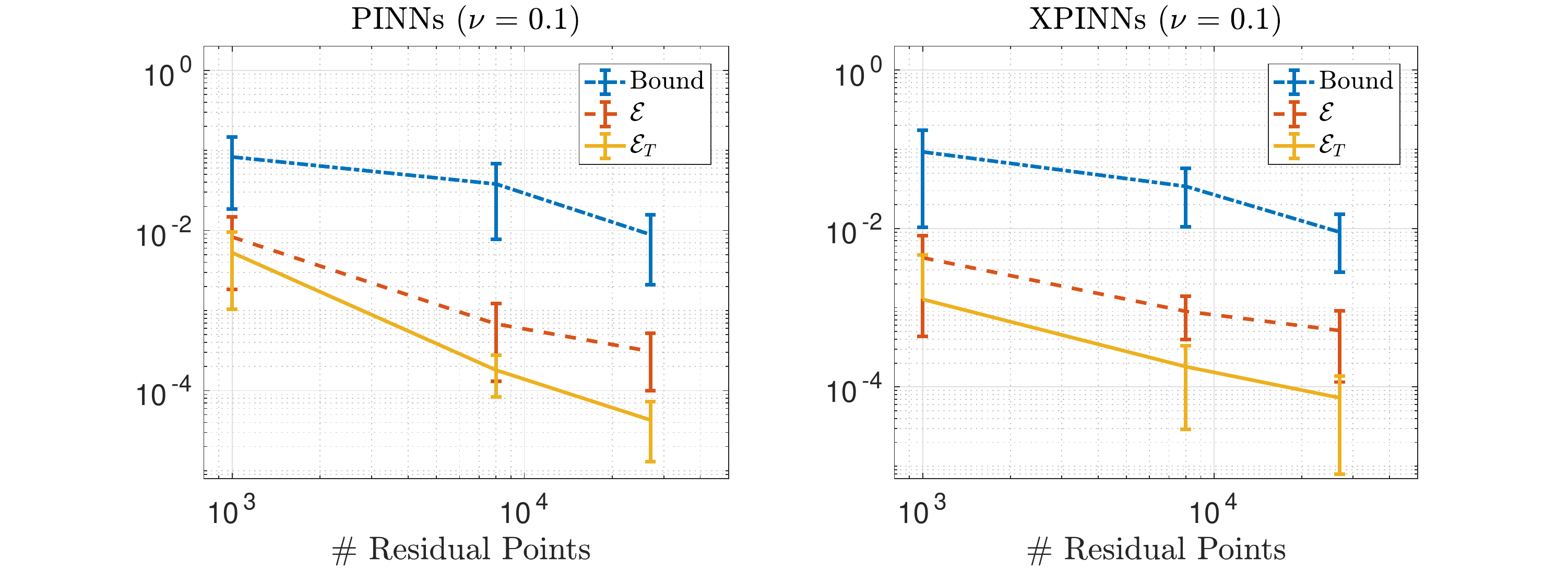}
\caption{Training and total errors for different number of quadrature points (residual points).}
\label{fig:fig2}
\end{figure}

We see from Figure \ref{fig:fig2} that both the training as well as total errors decay with respect to the number of quadrature points till they are saturated around $27$K quadrature points and do not decay any further. To further illustrate the error estimates derived in the previous section, we revisit the error estimate \eqref{eq:etrain1}. Given the elaboration of the appearing constants in \eqref{eq:bounds-constants} and the fact that we have access to the exact solution for the Taylor-Green vortex as well as to the (derivatives of) the PINN, we can \emph{explicitly compute} a theoretical bound on the total error in \eqref{eq:etrain1}. This error depends on the number of quadrature points as well as on the particular weights of the trained PINN. This theoretical bound is also depicted in Figure \ref{fig:fig2}. We see from this figure that the computed theoretical bound closely tracks the qualitative as well as quantitative behavior of the total error for all cases considered here. The rates of decay of both the error and the bound are very similar. However the bound is not quantitatively sharp as there is an approximately one order of magnitude difference in its amplitude vis a vis the total error. Such non-sharp bounds on the error are common in theoretical machine learning, see \citep{Arora} for instance. Even in the case of PINNs, they were already seen in \citep{MM1} where the authors observed at least two to three orders of magnitude discrepancy between their theoretical bounds and the realized total error. Given this context, an order of magnitude discrepancy between the bound in \eqref{eq:etrain1} and the observed total error is quite satisfactory.  
\begin{figure} [htpb] 
\centering
\includegraphics[trim=0cm 0cm 0cm 0cm, clip=true, scale=0.36, angle = 0]{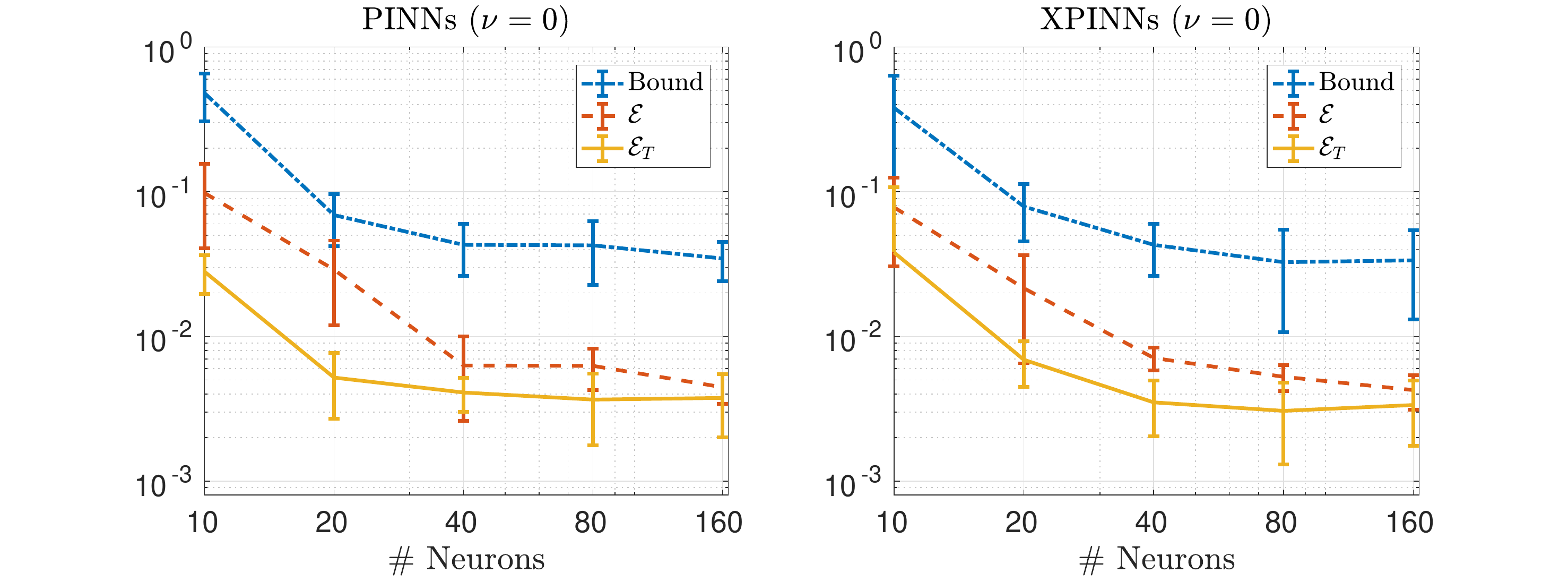}
\includegraphics[trim=0cm 0cm 0cm 0cm, clip=true, scale=0.36, angle = 0]{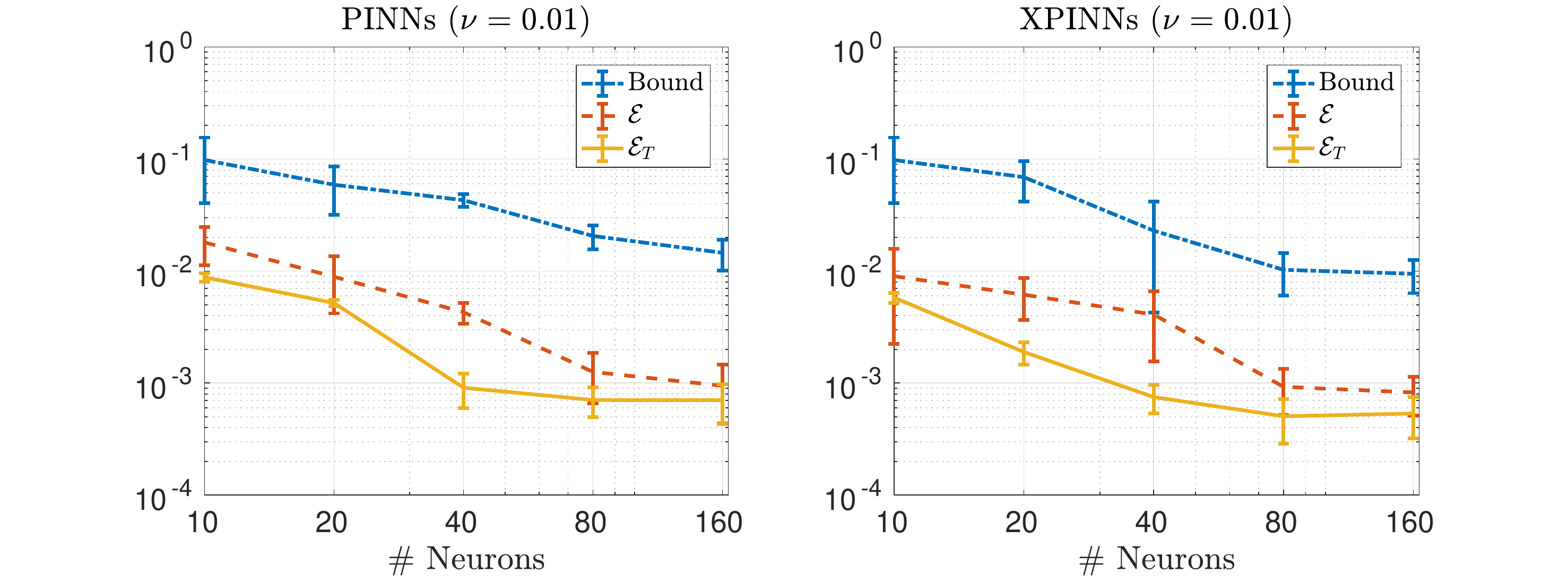}
\includegraphics[trim=0cm 0cm 0cm 0cm, clip=true, scale=0.36, angle = 0]{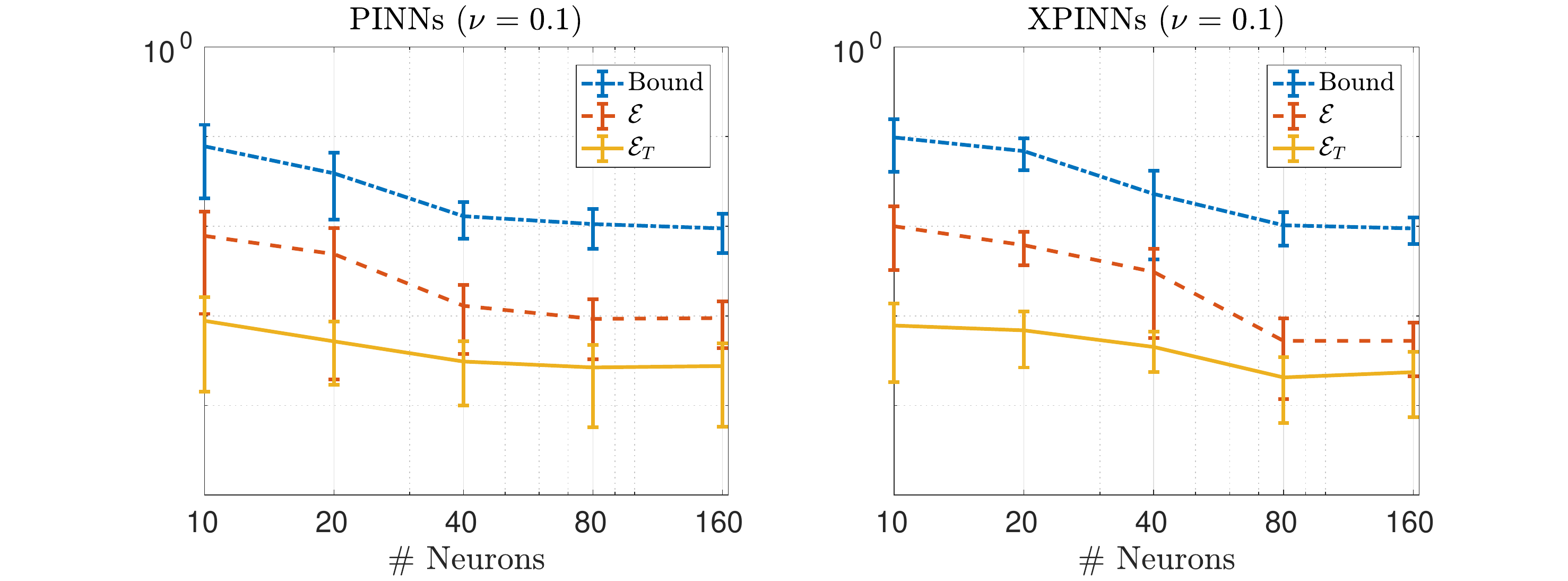}
\caption{Training and total errors for different number of neurons.}
\label{fig:fig3}
\end{figure}

Finally, we study the behavior of the error as the number of neurons is increased. Given our theoretical considerations, where the relevant error estimates where shown for tanh neural networks with two hidden layers, we restrict ourselves to this setting by only varying the network width and keep the number of hidden layers fixed at two. We again train the PINN and XPINNs networks 20 times with different parameter initializations. The learning rate has the same value as in the previous numerical experiments and the number of quadrature points is fixed at $64$ K. The resulting training and total errors are presented in Figure \ref{fig:fig3} and show that the total error decreases with the number of neurons in each layer till it gets saturated. Moreover, the computable upper bound \eqref{eq:etrain1} is also depicted and we see from this figure, that the bound \eqref{eq:etrain1} follows the same decaying trend, till saturation, as the total as well as training errors in this particular example. 
\section{Discussion}
\label{sec:discussion}
Physics-informed neural networks have been very successful in the numerical approximation of the solutions of forward as well as inverse problems for various classes of PDEs. However, there is a significant paucity of theoretical results on the resulting errors. Following the framework of a recent paper \citep{DRM1}, we revisit the key theoretical questions Q1 (on the smallness of the PDE residual in the class of neural networks), Q2 (a small residual implying a small total error) and Q3 (small training errors imply small total errors for sufficient number of quadrature points), raised in the introduction. We have answered these questions affirmatively for the incompressible Navier-Stokes equations in this paper. The incompressible Navier-Stokes equations constitute a very important example for nonlinear PDEs and PINNs have already been used in approximating them before \citep{jin2021nsfnets} but without much theoretical justification. 

Summarizing our theoretical results, we have shown in this paper that 
\begin{itemize}
    \item For sufficiently smooth (Sobolev regular) initial data, there exists neural networks, with the tanh activation function and with two hidden layers, such that the resulting PDE residuals can be arbitrarily small. Moreover in Theorem \ref{thm:pinn-approx}, we obtain very precise quantitative estimates on the sizes of resulting neural networks, in terms of regularity of the underlying classical solution. The proof of this approximation result relies heavily on the smoothness of the solutions of Navier-Stokes and on the approximation of smooth functions by neural networks in sufficiently high Sobolev norms. 
    \item In Theorem \ref{thm:stability}, we show that the total $L^2$ error of the PINN (and XPINN) approximations is bounded by the PDE residuals for the incompressible Navier-Stokes equations. Moreover, the underlying constants in the bound are clearly quantified in terms of the underlying classical solution as well as the approximating neural networks. This result leverages the stability (or rather coercivity) of classical solutions of the Navier-Stokes equations. Thus, we answer question Q2 affirmatively by showing a small PDE residual implies a small total error. 
    \item In Theorem \ref{thm:generalization}, we answer question Q3 by proving a bound \eqref{eq:etrain1} on the total error in terms of the training error and the number of quadrature points. Thus, if one reaches a global minimum of the underlying optimization problem \eqref{eq:opt}, one can show that the training error, and consequently the total error, can be made as small as possible if sufficient number of quadrature points are considered.
\end{itemize}
Taken together, the above theorems constitute the first comprehensive theoretical analysis of PINNs (and XPINNs) for a prototypical nonlinear PDE, the Navier-Stokes equations. We also illustrate the bounds in a simple numerical experiment demonstrating a qualitative as well as quantitative agreement between the rigorous bounds and the empirical results. 

Given this account of the strengths of our results, it is also fair to point out possible limitations and highlight avenues for future investigation. These include
\begin{itemize}
    \item Our estimates do not estimate the training error $\Et$, except under the assumption that one finds a global minimum for the optimization problem \eqref{eq:opt}. In practice, it is well known that (stochastic) gradient descent algorithms converge to local minima. In such cases, there is no guarantee on the smallness of the training error. Thus, one needs to find new techniques to estimate training errors for PINNs. On the other hand, our and other numerical results, see \citep{jin2021nsfnets} for instance, indicate that the training error can be made small. Then a bound like \eqref{eq:etrain1} clearly indicates that the overall error will be small. This is indeed borne out in numerical experiments (see Figure \ref{fig:fig3}). 
    \item Our estimates rely heavily on the regularity of the underlying solutions of Navier-Stokes equations \eqref{eq:navier-stokes}. There are two caveats in this context. First, in two space dimensions, one knows that the underlying solution will be sufficiently regular if the corresponding initial data is regular enough \citep{temam2001navier}. However, in three space dimensions, such results are a part of the millennium prize problems and are incredibly hard to obtain. On the more practical level, it is clear from Theorems \ref{thm:stability} and \ref{thm:generalization} that the errors will grow if the $C^1$ norms of the underlying exact solutions are large. This is clearly the case, particularly in three space dimensions, where solution gradients grow as vortices are stretched. Hence, one can expect the PINN errors to grow too and this is indeed seen in practice (see section 5 of \citep{MM1} for instance). However, traditional numerical methods such as finite element methods and spectral viscosity methods also suffer from the same issue and it is not expected to be different for PINNs. In this context, it would be interesting to investigate if the approaches which are based on weak formulations of PDE residuals might lead to better estimates and numerical results.  
\end{itemize}
Finally, only the forward problem is considered here. It would be interesting to extend the theoretical tools and bounds in the paper to inverse problems for the Navier-Stokes equations (see \citep{KAR4}) as well as physics-informed operator learning \citep{piol}.

\bibliographystyle{agsm}
\bibliography{ref}

@article{jagtap2022deep,
  title={Deep learning of inverse water waves problems using multi-fidelity data: Application to {Serre--Green--Naghdi} equations},
  author={Jagtap, Ameya D and Mitsotakis, Dimitrios and Karniadakis, George Em},
  journal={Ocean Engineering},
  volume={248},
  pages={110775},
  year={2022},
  publisher={Elsevier}
}

@article{jagtap2022physics,
  title={Physics-informed neural networks for inverse problems in supersonic flows},
  author={Jagtap, Ameya D and Mao, Zhiping and Adams, Nikolaus and Karniadakis, George Em},
  journal={arXiv preprint arXiv:2202.11821},
  year={2022}
}

@article{shukla2021physics,
  title={A Physics-Informed Neural Network for Quantifying the Microstructural Properties of Polycrystalline Nickel Using Ultrasound Data: A promising approach for solving inverse problems},
  author={Shukla, Khemraj and Jagtap, Ameya D and Blackshire, James L and Sparkman, Daniel and Karniadakis, George Em},
  journal={IEEE Signal Processing Magazine},
  volume={39},
  number={1},
  pages={68--77},
  year={2021},
  publisher={IEEE}
}

@article{hu2021extended,
  title={When Do Extended Physics-Informed Neural Networks ({XPINNs}) Improve Generalization?},
  author={Hu, Zheyuan and Jagtap, Ameya D and Karniadakis, George Em and Kawaguchi, Kenji},
  journal={arXiv preprint arXiv:2109.09444},
  year={2021}
}

@article{shukla2021parallel,
  title={Parallel physics-informed neural networks via domain decomposition},
  author={Shukla, Khemraj and Jagtap, Ameya D and Karniadakis, George Em},
  journal={Journal of Computational Physics},
  volume={447},
  pages={110683},
  year={2021},
  publisher={Elsevier}
}

@article{byrd1995limited,
  title={A limited memory algorithm for bound constrained optimization},
  author={Byrd, Richard H and Lu, Peihuang and Nocedal, Jorge and Zhu, Ciyou},
  journal={SIAM Journal on scientific computing},
  volume={16},
  number={5},
  pages={1190--1208},
  year={1995},
  publisher={SIAM}
}

@article{deryck2021approximation,
  title={On the approximation of functions by tanh neural networks},
  author={De Ryck, Tim and Lanthaler, Samuel and Mishra, Siddhartha},
  journal={Neural Networks},
  volume={143},
  pages={732--750},
  year={2021},
  publisher={Elsevier}
}

@article{guhring2021approximation,
  title={Approximation rates for neural networks with encodable weights in smoothness spaces},
  author={G{\"u}hring, Ingo and Raslan, Mones},
  journal={Neural Networks},
  volume={134},
  pages={107--130},
  year={2021},
  publisher={Elsevier}
}

@article{shin2020convergence,
  title={On the convergence and generalization of physics informed neural networks},
  author={Shin, Yeonjong and Darbon, Jerome and Karniadakis, George Em},
  journal={arXiv preprint arXiv:2004.01806},
  year={2020}
}

@article{mishra2020enhancing,
  title={Enhancing accuracy of deep learning algorithms by training with low-discrepancy sequences},
  author={Mishra, Siddhartha and Rusch, T Konstantin},
  journal={SIAM Journal on Numerical Analysis},
  volume={59},
  number={3},
  pages={1811--1834},
  year={2021},
  publisher={SIAM}
}

@article{majda2002vorticity,
  title={Vorticity and incompressible flow. Cambridge texts in applied mathematics},
  author={Majda, Andrew J and Bertozzi, Andrea L and Ogawa, A},
  journal={Appl. Mech. Rev.},
  volume={55},
  number={4},
  pages={B77--B78},
  year={2002}
}

@article{verfurth1999note,
  title={A note on polynomial approximation in {Sobolev} spaces},
  author={Verf{\"u}rth, R{\"u}diger},
  journal={ESAIM: Mathematical Modelling and Numerical Analysis},
  volume={33},
  number={4},
  pages={715--719},
  year={1999},
  publisher={EDP Sciences}
}

@book{naep,
  title={Numerical Methods for Elliptic and Parabolic Boundary Value Problems},
  author={Hiptmair, Ralf and Schwab, Christoph},
  year={2008},
  publisher={ETH Zürich}
}

@article{jin2021nsfnets,
  title={{NSFnets (Navier-Stokes flow nets): Physics-informed} neural networks for the incompressible {Navier-Stokes} equations},
  author={Jin, Xiaowei and Cai, Shengze and Li, Hui and Karniadakis, George Em},
  journal={Journal of Computational Physics},
  volume={426},
  pages={109951},
  year={2021},
  publisher={Elsevier}
}

@article{jagtap2020extended,
  title={Extended Physics-Informed Neural Networks ({XPINNs}): A Generalized Space-Time Domain Decomposition Based Deep Learning Framework for Nonlinear Partial Differential Equations},
  author={Jagtap, Ameya D and Karniadakis, George Em},
  journal={Communications in Computational Physics},
  volume={28},
  number={5},
  pages={2002--2041},
  year={2020},
  publisher={GLOBAL SCIENCE PRESS ROOM 3208, CENTRAL PLAZA, 18 HARBOUR RD, WANCHAI, HONG~…}
}

@book{temam2001navier,
  title={{Navier-Stokes} equations: theory and numerical analysis},
  author={Temam, Roger},
  volume={343},
  year={2001},
  publisher={American Mathematical Soc.}
}

@article{DLnat,
  title={Deep learning},
  author={LeCun, Yann and Bengio, Yoshua and Hinton, Geoffrey},
  journal={Nature},
  volume={521},
  number={7553},
  pages={436--444},
  year={2015},
  publisher={Nature Publishing Group},
}

@article{HEJ1,
  title={Deep learning-based numerical methods for high-dimensional parabolic partial differential equations and backward stochastic differential equations},
  author={E, Weinan and Han, Jiequn and Jentzen, Arnulf},
  journal={Communications in Mathematics and Statistics},
  volume={5},
  number={4},
  pages={349--380},
  year={2017},
  publisher={Springer},
}

@article{SZ1,
  title={Deep learning in high dimension: Neural network expression rates for generalized polynomial chaos expansions in {UQ}},
  author={Schwab, Christoph and Zech, Jakob},
  journal={Analysis and Applications},
  volume={17},
  number={01},
  pages={19--55},
  year={2019},
  publisher={World Scientific}
}

@article{Kuty,
  title={A theoretical analysis of deep neural networks and parametric {PDEs}},
  author={Kutyniok, Gitta and Petersen, Philipp and Raslan, Mones and Schneider, Reinhold},
  journal={Constructive Approximation},
  pages={1--53},
  year={2021},
  publisher={Springer}
}

@article{LMR1,
  title={Deep learning observables in computational fluid dynamics},
  author={Lye, Kjetil O and Mishra, Siddhartha and Ray, Deep},
  journal={Journal of Computational Physics},
  pages={109339},
  year={2020},
  publisher={Elsevier},
}

@article{LMPR1,
  title={Iterative surrogate model optimization ({ISMO}): An active learning algorithm for {PDE} constrained optimization with deep neural networks},
  author={Lye, Kjetil O and Mishra, Siddhartha and Ray, Deep and Chandrashekar, Praveen},
  journal={Computer Methods in Applied Mechanics and Engineering},
  volume={374},
  pages={113575},
  year={2021},
  publisher={Elsevier}
}

@article{ChenChen,
  title={Universal approximation to nonlinear operators by neural networks with arbitrary activation functions and its application to dynamical systems},
  author={Chen, Tianping and Chen, Hong},
  journal={IEEE Transactions on Neural Networks},
  volume={6},
  number={4},
  pages={911--917},
  year={1995},
  publisher={IEEE}
}

@article{DeepOnet,
  title={{DeepONet}: Learning nonlinear operators for identifying differential equations based on the universal approximation theorem of operators},
  author={Lu, Lu and Jin, Pengzhan and Karniadakis, George Em},
  journal={arXiv preprint arXiv:1910.03193},
  year={2019}
}

@article{LMK1,
  title={Error estimates for deeponets: A deep learning framework in infinite dimensions},
  author={Lanthaler, Samuel and Mishra, Siddhartha and Karniadakis, George E},
  journal={Transactions of Mathematics and Its Applications},
  volume={6},
  number={1},
  year={2022},
  publisher={Oxford University Press}
}

@misc{FNO,
      title={Fourier Neural Operator for Parametric Partial Differential Equations}, 
      author={Zongyi Li and Nikola Kovachki and Kamyar Azizzadenesheli and Burigede Liu and Kaushik Bhattacharya and Andrew Stuart and Anima Anandkumar},
      year={2020},
      eprint={2010.08895},
      archivePrefix={arXiv},
      primaryClass={cs.LG}
}

@article{DPT,
  title={Neural-network-based approximations for solving partial diﬀerential equations},
  author={MWMG Dissanayake and N Phan-Thien},
  journal={Communications in Numerical Methods in Engineering},
  year={1994},
  
}

@article{Lag1,
	title={Artificial neural networks for solving ordinary and partial differential equations},
	author={Lagaris, I. E and Likas, A and Fotiadis, D. I},
	journal={IEEE Transactions on Neural Networks},
	volume={9(5)},
	pages={987--1000},
	year={2000}
}

@article{Lag2,
  title={Neural-network methods for boundary value problems with irregular boundaries},
  author={Lagaris, Isaac E and Likas, Aristidis and Papageorgiou G. D.},
  journal={IEEE Transactions on Neural Networks},
  volume={11},
  pages={1041--1049},
  year={2000},
  publisher={IEEE},
}

@article{KAR1,
  title={Hidden physics models: Machine learning of nonlinear partial differential equations},
  author={Raissi, Maziar and Karniadakis, George Em},
  journal={Journal of Computational Physics},
  volume={357},
  pages={125--141},
  year={2018},
  publisher={Elsevier},
}

@article{KAR2,
author={M. Raissi and P. Perdikaris and G. E. Karniadakis},
journal={Journal of Computational Physics},
title={Physics-informed neural networks: A deep
learning framework for solving forward and inverse problems involving nonlinear partial
differential equations},
volume={378},
pages={686-707},
year={2019},
}

@article{KAR4,
  title={Hidden fluid mechanics: A {Navier-Stokes} informed deep learning framework for assimilating flow visualization data},
  author={Raissi, Maziar and Yazdani, Alireza and Karniadakis, George Em},
  journal={arXiv preprint arXiv:1808.04327},
  year={2018},
}

@article{jag2,
  title={Conservative physics-informed neural networks on discrete domains for conservation laws: Applications to forward and inverse problems},
  author={Jagtap, Ameya D and Kharazmi, Ehsan and Karniadakis, George Em},
  journal={Computer Methods in Applied Mechanics and Engineering},
  volume={365},
  pages={113028},
  year={2020},
  publisher={Elsevier}
}

@article{KAR6,
author={Z. Mao and A. D. Jagtap and G. E. Karniadakis},
journal={Computer Methods in Applied Mechanics and Engineering},
title={Physics-informed neural networks for
high-speed flows.},
volume={360},
pages={112789},
year={2020},
}

@article{KAR7,
author={G. Pang and L. Lu and G. E. Karniadakis},
journal={SIAM journal of Scientific computing},
title={{fPINNs}: Fractional physics-informed neural networks},
volume={41},
pages={A2603-A2626},
year={2019},
}

@article{KAR8,
  title={{B-PINNs}: Bayesian physics-informed neural networks for forward and inverse {PDE} problems with noisy data},
  author={Yang, Liu and Meng, Xuhui and Karniadakis, George Em},
  journal={Journal of Computational Physics},
  volume={425},
  pages={109913},
  year={2021},
  publisher={Elsevier}
}

@article{BKMM1,
  title={Physics Informed Neural Networks ({PINNs}) for approximating nonlinear dispersive {PDEs}},
  author={Bai, Genming and Koley, Ujjwal and Mishra, Siddhartha and Molinaro, Roberto},
  journal={arXiv preprint arXiv:2104.05584},
  year={2021}
}

@article{MM1,
  title={Estimates on the generalization error of physics informed neural networks ({PINNs}) for approximating {PDEs}},
  author={Mishra, Siddhartha and Molinaro, Roberto},
  journal={arXiv preprint arXiv:2006.16144},
  year={2020}
}

@article{MM2,
  title={Estimates on the generalization error of physics-informed neural networks for approximating a class of inverse problems for {PDEs}},
  author={Mishra, Siddhartha and Molinaro, Roberto},
  journal={IMA Journal of Numerical Analysis},
  year={2021}
}

@article{MM3,
  title={Physics informed neural networks for simulating radiative transfer},
  author={Mishra, Siddhartha and Molinaro, Roberto},
  journal={Journal of Quantitative Spectroscopy and Radiative Transfer},
  volume={270},
  pages={107705},
  year={2021},
  publisher={Elsevier}
}

@article{Zhang1,
  title={Error estimates of residual minimization using
neural networks for linear equations},
  author={Y. Shin and Z. Zhang and G. E. Karniadakis},
  journal={arXiv preprint arXiv:2010.08019},
  year={2020},
}

@inproceedings{Arora,
  author    = {Sanjeev Arora and
               Rong Ge and
               Behnam Neyshabur and
               Yi Zhang},
  title     = {Stronger Generalization Bounds for Deep Nets via a Compression Approach},
  booktitle = {Proceedings of the 35th International Conference on Machine Learning, {ICML}},
  series    = {Proceedings of Machine Learning Research},
  volume    = {80},
  pages     = {254--263},
  year      = {2018},
}

@unpublished{DRM1,
author={T. {De Ryck} and S. Mishra},
title={Error analysis for physics informed neural networks ({PINNs}) approximating {Kolmogorov PDEs}},
note={Preprint, available from arXiv:2106:14473},
year={2021},
}

@unpublished{piol,
author={S. Wang and P. Perdikaris},
title={Long-time integration of parametric evolution equations
with physics-informed deeponets.},
note={Preprint, available from arXiv:2106:05384},
year={2021},
}

@article{KLM1,
  title={On universal approximation and error bounds
for {Fourier Neural Operators}},
  author={N. Kovachki and S. Lanthaler and S. Mishra},
  journal={Journal of Machine Learning Research},
  volume={22},
  pages={1-76},
  year={2021}
}

@article{UTB1,
  title={Error estimates for deep learning methods in fluid dynamics},
  author={A. Biswas and J. Tian and S. Ulusoy},
  journal={arXiv preprint arXiv:2008.02844v1},
  year={2020},
}

@article{fornberg1988generation,
  title={Generation of finite difference formulas on arbitrarily spaced grids},
  author={Fornberg, Bengt},
  journal={Mathematics of computation},
  volume={51},
  number={184},
  pages={699--706},
  year={1988}
}

\appendix

\section{Notation and auxiliary results}\label{sec:notation}

This section provides an overview of the notation used in the paper and recalls some basic results on Sobolev spaces.

\subsection{Multi-index notation}\label{sec:multi-index}

For $d\in\mathbb{N}$, we call a $d$-tuple of non-negative integers $\alpha \in \N^d_0$ a multi-index. We write $\abs{\alpha} = \sum_{i=1}^d \alpha_i$, $\alpha! = \prod_{i=1}^d \alpha_i!$ and, for $x\in\mathbb{R}^d$, we denote by $x^\alpha =  \prod_{i=1}^d x_i^{\alpha_i}$ the corresponding multinomial. Given two multi-indices $\alpha, \beta\in \N^d_0$, we say that $\alpha \leq \beta$ if, and only if, $\alpha_i\leq \beta_i$ for all $i=1,\dots, d$. For a multi-index $\alpha$, we define the following multinomial coefficient
\begin{equation}
    \binom{\abs{\alpha}}{\alpha} = \frac{\abs{\alpha}!}{\alpha!}, 
\end{equation}
and, given $\alpha \le \beta$, we define a corresponding multinomial coefficient by
\begin{equation}
    \binom{\beta}{\alpha} = \prod_{i=1}^d \binom{\beta_i}{\alpha_i} = \frac{\beta !}{\alpha! (\beta-\alpha)!}.
\end{equation}
For $\Omega\subseteq \mathbb{R}^d$ and a function $f:\Omega\to\mathbb{R}$ we denote by 
\begin{equation}
    D^\alpha f= \frac{\partial^{\abs{\alpha}} f}{\partial x_1^{\alpha_1}\cdots \partial x_d^{\alpha_d}}
\end{equation}
the classical or distributional (i.e. weak) derivative of $f$. 

We will also encounter the set $P_{n,d} = \{\alpha\in  \mathbb{N}_0^d : \abs{\alpha} = n\}$, for which it holds that $\abs{P_{n,d} } = \binom{n+d-1}{n}$.

\subsection{Sobolev spaces}\label{sec:sobolev}

Let $d\in\mathbb{N}$, $k\in\mathbb{N}_0$, $1\leq p\leq \infty$ and let $\Omega \subseteq \mathbb{R}^d$ be open. We denote by $L^p(\Omega)$ the usual Lebesgue space and for we define the Sobolev space $W^{k,p}(\Omega)$ as
\begin{equation}
    W^{k,p}(\Omega) = \{f \in L^p(\Omega): D^\alpha f \in L^p(\Omega) \text{ for all } \alpha\in\mathbb{N}^d_0 \text{ with } \abs{\alpha}\leq k\}. 
\end{equation}
For $p<\infty$, we define the following seminorms on $W^{k,p}(\Omega)$, 
\begin{equation}
    \abs{f}_{W^{m,p}(\Omega)} = \left(\sum_{\abs{\alpha}= m}\norm{D^\alpha f}^p_{L^p(\Omega)}\right)^{1/p} \qquad \text{for } m=0,\ldots, k, 
\end{equation}
and for $p=\infty$ we define
\begin{equation}
    \abs{f}_{W^{m,\infty}(\Omega)} =\max_{\abs{\alpha}= m} \norm{D^\alpha f}_{L^\infty(\Omega)}\qquad \qquad \text{for } m=0,\ldots, k. 
\end{equation}
Based on these seminorms, we can define the following norm for $p<\infty$,
\begin{equation}
    \norm{f}_{W^{k,p}(\Omega)} = \left(\sum_{m=0}^k \abs{f}_{W^{m,p}(\Omega)}^p\right)^{1/p}, 
\end{equation}
and for $p=\infty$ we define the norm
\begin{equation}
    \norm{f}_{W^{k,\infty}(\Omega)} =\max_{0\leq m\leq k}  \abs{f}_{W^{m,\infty}(\Omega)}. 
\end{equation}
The space $W^{k,p}(\Omega)$ equipped with the norm $\norm{\cdot}_{W^{k,p}(\Omega)}$ is a Banach space. 

We denote by $C^k(\Omega)$ the space of functions that are $k$ times continuously differentiable and equip this space with the norm $\norm{f}_{C^k(\Omega)} = \norm{f}_{W^{k,\infty}(\Omega)}$.

We define the Hilbertian Sobolev spaces for $k\in\mathbb{N}_0$ as $H^k(\Omega)=W^{k,2}(\Omega)$ with corresponding norms $\norm{\cdot}_{H^k(\Omega)}=\norm{\cdot}_{W^{k,2}(\Omega)}$ and seminorms $\abs{\cdot}_{H^m(\Omega)}=\abs{\cdot}_{W^{m,2}(\Omega)}$ for integers $m$ with $0\leq m\leq k$. If $k$ is large enough, the space $H^k(\Omega)$ is a Banach algebra. We also recall a version of the Sobolev embedding theorem and a multiplicative trace inequality. 

\begin{lemma}\label{lem:banach-algebra}
For $d,k\in\mathbb{N}$ with $k>\frac{d}{2}$, $H^k(\Omega)$ is a Banach algebra i.e., there exists $c_k>0$ such that
\begin{equation}
    \forall u,v\in H^k(\Omega): \: \norm{uv}_{H^k(\Omega)}\leq c_k\norm{u}_{H^k(\Omega)}\norm{v}_{H^k(\Omega)}.
\end{equation}
\end{lemma}

\begin{lemma}\label{lem:sobolev-embedding}
Let $d\in\mathbb{N}$, $k,\ell\in\mathbb{N}_0$ with $k>\ell+\frac{d}{2}$ and $\Omega\subset \mathbb{R}^d$ an open set. Every function $f\in H^k(\Omega)$ has a continuous representative belonging to $C^\ell(\Omega)$. 
\end{lemma}

\begin{lemma}[Multiplicative trace inequality, e.g. Theorem 3.10.1 in \citep{naep}]\label{lem:trace-inequality}
Let $d\geq 2$, $\Omega\subset \mathbb{R}^d$ be a Lipschitz domain and let $\gamma_0:H^1(\Omega)\to L^2(\partial \Omega): u\mapsto u\vert_{\partial \Omega}$ be the trace operator. Denote by $h_\Omega$ the diameter of $\Omega$ and by $\rho_\Omega$ the radius of the largest $d$-dimensional ball that can be inscribed into $\Omega$. Then it holds that
\begin{equation}
    \norm{\gamma_0 u}_{L^2(\partial \Omega)} \leq \sqrt{\frac{2\max\left\{2h_\Omega,d\right\}}{\rho_\Omega}}\norm{u}_{H^1(\Omega)}
\end{equation}
\end{lemma}

Next, we recall the Bramble-Hilbert lemma, which quantifies the accuracy of polynomial approximations of functions in Sobolev spaces. We present a variant of the Bramble-Hilbert lemma for Hilbertian Sobolev spaces proven in \citep{verfurth1999note}.

\begin{lemma}\label{lem:BH}
Let $\Omega$ be a bounded convex open domain $\mathbb{R}^d$, $d\geq 2$, with diameter $h$. For every $f\in H^m(\Omega)$ there exists a polynomial $p$ of degree at most $m-1$ such that for all $0\leq j\leq m-1$ it holds that
\begin{equation}
    \abs{f-p}_{H^j(\Omega)} \leq c_{m,j} h^{m-j}\abs{f}_{H^m(\Omega)}
\end{equation}
where
\begin{equation}
   c_{m,j} = \pi^{j-m}\binom{d+j-1}{j}^{1/2} \frac{((m-j)!)^{1/2}}{\left(\left\lceil\frac{m-j}{d}\right\rceil!\right)^{d/2}}.
\end{equation}
\end{lemma}

We proceed by stating a corollary of the general Leibniz rule for Sobolev regular functions. 

\begin{lemma}\label{lem:leibniz}
Let $d\in\mathbb{N}$, $k\in\mathbb{N}_0$, $\Omega\subset \mathbb{R}^d$ and $f\in H^k(\Omega)$ and $g\in W^{k,\infty}(\Omega)$. Then it holds that
\begin{equation}
    \norm{fg}_{H^k}\leq 2^k \norm{f}_{H^k}\norm{g}_{W^{k,\infty}}.
\end{equation}
\end{lemma}

Finally, we present a result on the Sobolev norm of the composition of two $n$ times continuously differentiable functions \cite[Lemma A.7]{deryck2021approximation}. 

\begin{lemma}\label{lem:faa-di-bruno}
Let $d,m,n\in\mathbb{N}$, $\Omega_1\subset \mathbb{R}^d$, $\Omega_2\subset \mathbb{R}^m$, $f\in C^n(\Omega_1; \Omega_2)$ and $g\in C^n(\Omega_2; \mathbb{R})$. Then it holds that 
\begin{equation}
    \norm{g \circ f}_{W^{n,\infty}(\Omega_1)} \leq 16(e^2n^{4}md^2)^{n} \norm{g}_{W^{n,\infty}(\Omega_2)} \max_{1\leq i\leq m}\norm{(f)_i}_{W^{n,\infty}(\Omega_1)}^n.
\end{equation}
\end{lemma}

{\color{black}
\section{Function approximation by tanh neural networks}\label{sec:tanh}
}
In this section, we show how one can prove that for every $f\in H^m(\Omega)$, $m\geq 3$, there exists a tanh neural network $\hat{f}$ with two hidden layers such that $ \Vert f-\hat{f}\Vert_{H^2(\Omega)}\leq \epsilon$ for some $\epsilon>0$. Results of this type can be found in \citep{guhring2021approximation} for very general activation functions and in \citep{deryck2021approximation} for the tanh activation function. Both references prove such a result as follows: first, one divides the domain $\Omega$ into cubes of edge length $1/N$, with $N\in\N$ large enough. On each of these cubes, $f$ can be  approximated in Sobolev norm by a polynomial, by virtue of the Bramble-Hilbert lemma. A global approximation can then be constructed by multiplying each polynomial with the indicator function of the corresponding cubes and summing over all cubes. Replacing these polynomials, multiplications and indicator functions with suitable neural networks results in a new approximation that has approximately the same accuracy. 

In the following, we choose \citep{deryck2021approximation} as a guideline, as it provides explicit upper bounds on the neural network size, which is something we aim to provide for the Navier-Stokes equations. The preceding reference \cite{guhring2021approximation} does not give such explicit bounds, but one can use their proofs to obtain similar explicit bounds for more general activation functions.\footnote{A more complete discussion about the differences between \cite{guhring2021approximation} and \cite{deryck2021approximation} can be found in \cite{deryck2021approximation}.} We improve upon \citep{deryck2021approximation} by adapting their proof of the neural network approximation of polynomials such that the bound on the network weights grows less fast. This is accomplished by using an $n$-th order accurate finite difference formula in the proof, rather than a second order accurate one. Below, we provide an overview of the improved versions of the results of \citep{deryck2021approximation}.  

The following lemma treats the neural network approximation of multivariate monomials and is the main source of change compared to the original results in \citep{deryck2021approximation}. All updates in the other results are mainly consequences of the following lemma. 

\begin{lemma}[Approximation of multivariate monomials]\label{lem:pol-tanh}
Let $d,s,n\in \mathbb{N}$, $k\in \mathbb{N}_0$ and $M>0$. Then for every $\epsilon>0$, there exists a shallow tanh neural network $\Phi_{s,d}:[-M,M]^d\to\mathbb{R}^{\abs{P_{s,d+1}}}$ of width $3\left\lceil\frac{s+n-1}{2}\right\rceil\abs{P_{s,d+1} }$ such that
\begin{equation}
   \max_{\beta\in P_{s,d+1}} \norm{x^\beta - (\Phi_{s,d}(x))_{\iota(\beta)}}_{W^{2,\infty}([-M,M]^d)} \leq \epsilon,
\end{equation}
where $\iota:P_{s,d+1}\to\{1, \ldots \abs{P_{s,d+1}}\}$ is a bijection. Furthermore, the weights of the network scale as $O\left(\epsilon^{-s/n}\right)$ for small $\epsilon$.
\end{lemma}

\begin{proof}
We start by constructing a neural network $\fhat_{p,h,n}$ that approximates the univariate monomial $f_p:[-M,M]\to \R: x\mapsto x^p$ in $W^{k,\infty}$-norm. In \cite[Lemma 3.1]{deryck2021approximation} this has been done by using a second-order accurate finite difference formula. We will generalize this result by using an $n$-th-order accurate finite difference scheme. In particular we define $\fhat_{p,h,n}$ by, 
\begin{equation}
    \fhat_{p,h,n}(x) = \frac{1}{\sigma^{(p)}(0) h^p} \sum_{i=-\ell}^\ell a_i \sigma(ihx),
\end{equation}
where $\ell = \frac{p+n-1}{2}$ and where the $a_i$ are the solution to the system of equations
\begin{equation}
    \sum_{i=-\ell}^\ell a_i \ell^j = p!\, \delta (l-j)
    =
    \begin{cases}
    p! & (j=p), \\
    0 & (j\ne p),
    \end{cases}
    \quad \text{for } 0 \leq j \leq p+n-1. 
\end{equation}
A solution to this system exists and can even be efficiently constructed \citep{fornberg1988generation}. Following the exact steps of \cite[Lemma 3.1]{deryck2021approximation}, but now using the above equation instead of equation (18) in \citep{deryck2021approximation} we find that that for all $1\leq p \leq s$, $p$ odd, we can find neural networks $\fhat_{p,h,n}$ such that for arbitrary $k\in \N$ it holds,
\begin{equation}
     \ck{f_p- \hat{f}_{p,h}} \leq C(\sigma,M,s,n,k) h^n =: \epsilon. 
\end{equation} 
Note that $\{\fhat_{p,h,n}\::\: 1\leq p \leq s, \: p \text{ odd}\}$ is a shallow tanh neural network with $\ell = \frac{p+n-1}{2}$ neurons (where we used the symmetry of $\sigma$) and of which the weights grow as $\bigO(h^{-s}) = \bigO(\epsilon^{-s/n})$ for $\epsilon\to 0$. One can then follow the exact same steps of \cite[Lemma 3.2 and Section 3.2]{deryck2021approximation} to generalize this result to multivariate polynomials of arbitrary degree, which leads to the statement of this lemma. 
\end{proof}

\begin{lemma}[Shallow approximation of multiplication of $d$ numbers]\label{lem:mult-shallow}
Let $d,n\in \mathbb{N}$, $k\in \mathbb{N}_0$ and $M>0$. Then for every $\epsilon>0$, there exists a shallow tanh neural network $\widehat{\times}_d^\epsilon: [-M,M]^d\to\mathbb{R}$ of width $3\left\lceil\frac{d+n-1}{2}\right\rceil\abs{P_{d,d} }$ such that
\begin{equation}
   \ck{\widehat{\times}_d^\epsilon(x)-\prod_{i=1}^d x_i} \leq \epsilon.
\end{equation}
Furthermore, the weights of the network scale as $O(\epsilon^{-d/n})$.
\end{lemma}
\begin{proof}
This is the counterpart of \citep[Corollary 3.7]{deryck2021approximation}, with the only difference that now the construction of Lemma \ref{lem:pol-tanh} is used. 
\end{proof}

\begin{lemma}\label{lem:bound-der-tanh}
It holds that $\max\{\abs{\sigma(x)},\abs{\sigma'(x)}, \abs{\sigma''(x)}\}\leq 1$ for all $x\in\mathbb{R}$.
\end{lemma}

Next, we summarize the construction of an approximate partition of unity of a domain $\Omega = \prod_{i=1}^d[0,b_i]$, as in \cite[Section 4]{deryck2021approximation}. We divide the domain into cubes of edge length $1/N$ and denote the corresponding index set by 
\begin{equation}
    \mathcal{N}^N = \{j\in\mathbb{N}^d\::\: j_i\leq Nb_i \text{ for all }1\leq i\leq d\}.
\end{equation}
We can then define the cubes for every $j\in \mathcal{N}^N$ as,
\begin{equation}
    I_j^N = \bigtimes_{i=1}^d \left((j_i-1)/N,j_i/N\right). 
\end{equation}
Observe that $\abs{\sigma'}$ and $\abs{\sigma''}$ are monotonously decreasing on $[1,\infty)$. Given $\epsilon > 0$, we first find an $\alpha = \alpha(N,\epsilon)$ large enough such that
\begin{align}\label{eq:alpha}
    \alpha/N \geq 1, \quad 1 - \sigma(\alpha/N) \leq \epsilon, \quad \alpha^m \abs{\sigma^{(m)}(\alpha/N)} \leq \epsilon \text{  for } m=1,2. 
\end{align}
A suitable choice of $\alpha$ is given by the following lemma. 

\begin{lemma}\label{lem:alpha-growth}
The conditions stated in \eqref{eq:alpha} for $0<\epsilon<1$ are satisfied if
\begin{equation}
    \alpha = N \ln(\frac{4N^2}{e^2\epsilon}).
\end{equation}
\end{lemma}
\begin{proof}
This is an adaptation of Lemma A.5 in \citep{deryck2021approximation} for $k=2$. The proof is as in \citep{deryck2021approximation}, except that one can use Lemma \ref{lem:bound-der-tanh} instead of \cite[Lemma A.4]{deryck2021approximation}.
\end{proof}

For $y\in \mathbb{R}$, we then define
\begin{align}\label{eq:pou-def}
\begin{split}
    \rho_1^N(y) &= \frac{1}{2}-\frac{1}{2}\sigma\left(\alpha\left(y-\frac{1}{N}\right)\right),\\
    \rho_j^N(y) &= \frac{1}{2}\sigma\left(\alpha\left(y-\frac{j-1}{N}\right)\right) - \frac{1}{2}\sigma\left(\alpha\left(y-\frac{j}{N}\right)\right)\quad \text{for } 2\leq j \leq N-1,\\
    \rho_N^N(y) &= \frac{1}{2}\sigma\left(\alpha\left(y-\frac{N-1}{N}\right)\right)+\frac{1}{2}.
    \end{split}
\end{align}
Finally, we define for $D\leq d$ the functions
\begin{equation}
    \Phi^{N,D}_j(x) = \prod^{D}_{i=1} \rho_{j_i}^{N_i}(x_i)
\end{equation}
and the sets $\mathcal{V}_D = \{v\in\mathbb{Z}^d: \max_{1\leq i \leq D}\abs{v_i}\leq 1 \text{ and } v_{D+1}=\cdots = v_d = 0\}$. The functions $\Phi^{N,d}_j$ approximate a partition of unity in the sense that for every $j$ it holds on $I_j^N$ that,
\begin{equation}
    \sum_{v\in\mathcal{V}_d}\Phi^{N,d}_{j+v} \approx 1 \quad \text{and} \quad \sum_{\substack{v\not\in\mathcal{V}_d,\\ j+v \in \{1,\ldots, N\}^d}}\Phi^{N,d}_{j+v} \approx 0.
\end{equation}
This is made exact in the following lemmas. 

\begin{lemma}[Lemma 4.1 in \citep{deryck2021approximation}]\label{lem:indicator-close}
If $k\in \mathbb{N}_0$ and $0<\epsilon < 1/4$, then 
\begin{equation}
     \norm{\sum_{v\in \mathcal{V}_d}\Phi^{N,d}_{j+v}-1}_{W^{k,\infty}(I_j^N)} \leq 2^{kd} d\epsilon.
\end{equation}
\end{lemma}

\begin{lemma}\label{lem:indicator-far}
Let $k\in\{0,1,2\}$ and $v\in\mathbb{Z}^d$ with $\norm{v}_\infty\geq 2$. Then it holds that
\begin{equation}
     \norm{\Phi^{N,d}_{j+v}}_{W^{k,\infty}(I_j^N)} \leq \alpha^k \epsilon.
\end{equation}
\end{lemma}
\begin{proof}
This is an adaptation of Lemma 4.2 in \citep{deryck2021approximation} for 
$k\leq 2$. The proof is as in \citep{deryck2021approximation}, except that one can use Lemma \ref{lem:bound-der-tanh} instead of \cite[Lemma A.4]{deryck2021approximation}.
\end{proof}

{\color{black} We can now present a generalization of \cite[Theorem 5.1]{deryck2021approximation} where a parameter $n$ can be freely chosen in order to control the network width and weights. For $n=2$ one recovers \cite[Theorem 5.1]{deryck2021approximation} exactly. }

\begin{theorem}\label{thm:tanh-approximation}
Let $d,n\geq 2$, $m\geq 3$, $\delta>0$, $a_i, b_i \in \mathbb{Z}$ with $a_i<b_i$ for $1\leq i\leq d$, $\Omega = \prod_{i=1}^d[a_i,b_i]$ and $f\in H^{m}(\Omega)$. Then for every $N\in\mathbb{N}$ with $N>5$ there exists a tanh neural network $\widehat{f}^N$ with two hidden layers, one of width at most $3\left\lceil\frac{m+n-2}{2}\right\rceil\abs{P_{m-1,d+1}}+\sum_{i=1}^d(b_i-a_i)(N-1)$ and another of width at most $3\left\lceil\frac{d+n}{2}\right\rceil\abs{P_{d+1,d+1}}N^d\prod_{i=1}^d (b_i-a_i)$, such that for $k\in\{0,1,2\}$ it holds that,
\begin{equation}
    \hkunit{f-\widehat{f}^N} \leq 2^k3^d C_{k,m,d,f} \left(1+\delta\right)\ln^k\left(\beta_{k,\delta, d,f} N^{d+m+2}\right)N^{-m+k} , 
\end{equation}
and where we define 
\begin{align}
    \beta_{k,\delta, d,f} &= \frac{5\cdot 2^{kd}\max\{\prod_{i=1}^d (b_i-a_i),d\}\max\{\ckunit{f},1\}}{3^d\delta\min\{1,C_{k,m,d,f}\} },\\
   C_{k,m,d,f} &=   \max_{0\leq \ell\leq k} \binom{d+\ell-1}{\ell}^{1/2} \frac{((m-\ell)!)^{1/2}}{\left(\left\lceil\frac{m-\ell}{d}\right\rceil!\right)^{d/2}}\left(\frac{3\sqrt{d}}{\pi }\right)^{m-\ell}\abs{f}_{H^m}. 
\end{align}
Moreover, the weights of $\widehat{f}^N$ scale as $O(N\ln(N)+N^\gamma)$ with $\gamma = \max\{m^2,d(2+m+d)\}/n$.
\end{theorem}

\begin{proof}

\textbf{Step 1: construction of the approximation. }
We divide the domain $\Omega$ into cubes of edge length $1/N$ and denote the corresponding index set by 
\begin{equation}
    \mathcal{N}^N = \{j\in\mathbb{N}^d\::\: j_i\leq N(b_i-a_i) \text{ for all }1\leq i\leq d\}.
\end{equation}
Furthermore we write $T = \prod_{i=1}^d (b_i-a_i)$. As a result, $\abs{\mathcal{N}^N}=TN^d$. Let us denote $J_j^N = \bigtimes_{i=1}^d \left((j_i-2)/N,(j_i+1)/N\right)$. We calculate that $\text{diam}(J_j^N) = \frac{3\sqrt{d}}{N}$. As a consequence, the Bramble-Hilbert lemma (Lemma \ref{lem:BH}) ensures the existence of a polynomial $p_j^N$ of degree at most $m-1$ such that for all $0\leq \ell\leq m-1$ it holds that
\begin{align}\label{eq:pkn-acc-sobolev}
\begin{split}
    \abs{f-p_j^N}_{H^\ell(J_j^N)} &\leq  \binom{d+\ell-1}{\ell}^{1/2} \frac{((m-\ell)!)^{1/2}}{\left(\left\lceil\frac{m-\ell}{d}\right\rceil!\right)^{d/2}}\left(\frac{3\sqrt{d}}{\pi N}\right)^{m-\ell}\abs{f}_{H^m}  =: \frac{\mathcal{C^*_{\ell}}}{N^{m-\ell}}.
\end{split}
\end{align}
To simplify notation, we also define $\mathcal{C}_k:=\max_{0\leq \ell \leq k}\mathcal{C}^*_\ell$ and $p^N=\sum_j p_j^N \chi_j$, where $\chi_j$ denotes the indicator function on $I_j^N$. 
Next, let $q_j^N$ be a tanh neural network as in Lemma \ref{lem:pol-tanh} (where we still leave $n\in\N$ undefined for the moment) such that 
\begin{equation}\label{eq:qkn-acc-sobolev}
    \norm{q_j^N-p_j^N}_{W^{k,\infty}(\Omega)} \leq \eta\quad \text{and}\quad \norm{q_j^N-p_j^N}_{H^{k}(\Omega)} \leq \eta. 
\end{equation}
In addition, we define 
\begin{equation}
    q_j^N(x)\widehat{\times} \Phi^{N,d}_j(x) := \widehat{\times}_{d+1}^{h}(q_j^N(x),\phi_{j_1}^{N,d}(x_1), \ldots, \phi_{j_d}^{N,d}(x_d)),
\end{equation}
where $\widehat{\times} := \widehat{\times}_{d+1}^{h}$ is the network from Corollary \ref{lem:mult-shallow} and $h=h(N)$ will be defined in the remainder of the proof. 
We then define our approximation as
\begin{equation}\label{eq:tanh-approx-def}
    \widehat{f}^N(x) = \sum_{j\in\mathcal{N}^N} q_j^N(x)\widehat{\times} \Phi^{N,d}_j(x). 
\end{equation}

\textbf{Step 2: estimating the error of the approximation. }The triangle inequality gives us
\begin{align}\label{eq:three-terms-sobolev}
\begin{split}
    \hkunit{f-\widehat{f}^N} &\leq \hkunit{f-\sum_{j\in\mathcal{N}^N} f \cdot \Phi^{N,d}_j} +  \hkunit{\sum_{j\in\mathcal{N}^N} (f -q_j^N)\cdot \Phi^{N,d}_j}\\
    &+ \hkunit{\sum_{j\in\mathcal{N}^N} (q_j^N \cdot \Phi^{N,d}_j-q_j^N \widehat{\times} \Phi^{N,d}_j)}
    \end{split}
\end{align}
We proceed by bounding each term of the right hand side separately. 

\textit{Step 2a: First term of \eqref{eq:three-terms-sobolev}.} Let $i\in\mathcal{N}^N$ be arbitrary. 
Recalling that $\mathcal{V}_d = \{v\in\mathbb{Z}^d: \norm{v}_\infty\leq 1\}$, we observe that for $k\in\{0,1,2\}$,
\begin{align}
    \begin{split}
        \hki{f-\sum_{j\in\mathcal{N}^N} f \cdot \Phi^{N,d}_j} &\leq 2^k \hki{f}\cki{1-\sum_{v\in \mathcal{V}_d} \Phi^{N,d}_{i+v}} \\ & \quad +2^k \hki{f}\cki{ \sum_{\substack{j\in\mathcal{N}^N \\ j-i\not\in\mathcal{V}_d}} \Phi^{N,d}_{j}}\\
        &\leq 2^k \hki{f} (2^{kd}d\epsilon + \abs{\mathcal{N}^N} \alpha^k \epsilon)\\
        &\leq 2^{k(1+d)} \hki{f} d\epsilon \\
        & \quad + 2^k\hki{f} \abs{\mathcal{N}^N} N^k \ln^k\left(\frac{4N^2}{e^2\epsilon}\right)\epsilon\\
        & \leq 2^k 3^d \frac{\delta}{4} \ln^k\left(\frac{4N^2}{e^2\epsilon}\right) \frac{\mathcal{C}_k}{N^{m-k}},  
    \end{split}
\end{align}
where we used Lemma \ref{lem:leibniz} with $k=2$, Lemma \ref{lem:indicator-close}, Lemma \ref{lem:indicator-far} and Lemma \ref{lem:alpha-growth}, as well as a suitable definition of $\epsilon$, e.g. satisfying
\begin{align}\label{eqn:def-eps-1}
\epsilon 
\le
\frac{
3^d\delta \mathcal{C}_k
}{
2^{3+k+kd}\max\{T,d\}N^{d+m}\hkunit{f}
},
\end{align}
where we used that $N>5$.

\textit{Step 2b: Second term of \eqref{eq:three-terms-sobolev}.} Let $\beta\in\mathbb{N}_0^d$ be such that $\abs{\beta}\leq k$. Then as a consequence of the general Leibniz rule we find that
\begin{equation}
    \norm{D^\beta \left(\sum_{v\in \mathcal{V}_d} (f-q_{i+v}^N) \Phi^{N,d}_{i+v}\right)}_{L^2(I_i^N)} \leq \sum_{\beta'\leq \beta}\binom{\beta}{\beta'} \sum_{v\in \mathcal{V}_d} \norm{D^{\beta'}(f-q_{i+v}^N)}_{L^2(I_i^N)}\norm{D^{\beta-\beta'} \Phi^{N,d}_{i+v}}_{L^\infty(I_i^N)}.
\end{equation}
For every $v\in \mathcal{V}_d$ and $\beta'\leq \beta$ with $\ell := \abs{\beta-\beta'}$, we can then use the bounds
\begin{equation}
    \norm{D^{\beta'}(f-q_{i+v}^N)}_{L^2(I_i^N)} \leq \norm{f-q_{i+v}^N}_{H^{k-\ell}(I^N_i)} \leq \frac{\mathcal{C}_k}{N^{m-k+\ell}} + \eta,
\end{equation}
which follows from \eqref{eq:pkn-acc-sobolev} and \eqref{eq:qkn-acc-sobolev}, and, 
\begin{equation}
   \norm{D^{\beta-\beta'} \Phi^{N,d}_{i+v}}_{L^\infty(I_i^N)} \leq  N^\ell \ln^\ell\left(\frac{4N^2}{e^2\epsilon}\right),
\end{equation}
which follows from Lemma \ref{lem:bound-der-tanh} and Lemma \ref{lem:alpha-growth}.
As $\sum_{\beta'\leq \beta}\binom{\beta}{\beta'} \leq 2^k$ (as a consequence of the multi-binomial theorem), we find that
\begin{align}
\begin{split}
    \cki{\sum_{v\in \mathcal{V}_d} (f-q_{i+v}^N) \Phi^{N,d}_{i+v}} &\leq 2^k 3^d \left(\frac{\mathcal{C}_k}{N^{m-k}} + \eta N^k\right) \ln^k\left(\frac{4N^2}{e^2\epsilon}\right).
\end{split}
\end{align}
Combining this result with the triangle inequality, Lemma \ref{lem:alpha-growth}, Lemma \ref{lem:leibniz}, \eqref{eq:pkn-acc-sobolev}, \eqref{eq:qkn-acc-sobolev}, Lemma \ref{lem:indicator-far} and the fact that $\ln(x)\leq \sqrt{x}$ for $x>0$, we find that
\begin{align}
    \begin{split}
       &\hki{\sum_{j\in\mathcal{N}^N} (f -q_j^N)\cdot \Phi^{N,d}_j}\\ &\leq \hki{\sum_{v\in \mathcal{V}_d} (f-q_{i+v}^N) \Phi^{N,d}_{i+v}} \quad+ \sum_{\substack{j\in\mathcal{N}^N \\ j-i\not\in\mathcal{V}_d}} \hki{(f-q_{j}^N) \Phi^{N,d}_{j}}\\
       &\leq \hki{\sum_{v\in \mathcal{V}_d} (f-q_{i+v}^N) \Phi^{N,d}_{i+v}} \quad+ \sum_{\substack{j\in\mathcal{N}^N \\ j-i\not\in\mathcal{V}_d}} 2^k \hki{(f-q_{j}^N)} \cki{\Phi^{N,d}_{j}}\\
       &\leq 2^k3^d \left(\frac{\mathcal{C}_k}{N^{m-2}} + \eta N^k\right) \ln^k\left(\frac{4N^2}{e^2\epsilon}\right)  + 2^k \abs{\mathcal{N}^N} \left(\mathcal{C}_k+\eta \right) N^k \ln^k\left(\frac{4N^2}{e^2\epsilon}\right){\epsilon} \\
       &\leq 2^k3^d \left(1+\frac{\delta}{4}\right) \ln^k\left(\frac{4N^2}{e^2\epsilon}\right) \frac{\mathcal{C}_k}{N^{m-k}},
    \end{split}
\end{align}
where we obtain the last inequality by making a suitable choice of $\eta$ and $\epsilon$, satisfying
\begin{equation}\label{eqn:def-eps-2}
    \eta \leq \frac{\delta \mathcal{C}_k}{8N^{m}} \quad \text{and} \quad \epsilon \leq \frac{3^d\delta }{4TN^{d+m}}. 
\end{equation}

\textit{Step 2d: Third term of \eqref{eq:three-terms-sobolev}.} Finally, using the triangle inequality, Lemma \ref{lem:faa-di-bruno}, Lemma \ref{lem:pol-tanh} and Lemma \ref{lem:bound-der-tanh} we obtain that for some $C>0$ depending only on $k$ and $d$, 
\begin{align}
    \begin{split}
&\hki{\sum_{j\in\mathcal{N}^N} (q_j^N \cdot \Phi^{N,d}_j-q_j^N \widehat{\times} \Phi^{N,d}_j)} \leq \sqrt{\mu(\Omega)} \cki{\sum_{j\in\mathcal{N}^N} (q_j^N \cdot \Phi^{N,d}_j-q_j^N \widehat{\times} \Phi^{N,d}_j)}\\&\leq \sqrt{\mu(\Omega)}\abs{\mathcal{N}^N} C\cdot \ck{\widehat{\times}_{d+1}^{h} 
{\color{black}- \prod_{i=1}^{d+1} x_i}
}\left(\ckunit{q_j^N} + \ckunit{\rho^N_i}\right)^k\\
&\leq \sqrt{\mu(\Omega)}\abs{\mathcal{N}^N} C\cdot h \left(\ckunit{q_j^N} + \alpha^k\right)^k\\
&\leq 2^k 3^d \frac{\delta}{4} \ln^k\left(\frac{4N^2}{e^2\epsilon}\right) \frac{\mathcal{C}_k}{N^{m-k}},
    \end{split}
\end{align}

where we obtain the last inequality by {\color{black}making a suitable choice of $h$, satisfying}
\begin{equation}\label{eqn:def-h}
   h \leq \frac{3^d \delta \mathcal{C}_k}{4\sqrt{\mu(\Omega)}TN^{d+m-k}C\left(\ckunit{q_j^N} + \alpha^k\right)^k}.
\end{equation}

\textit{Step 2e: Final error bound.} From \eqref{eqn:def-eps-1} and \eqref{eqn:def-eps-2} we find that a suitable definition of $\epsilon$ is given by
\begin{equation}\label{eqn:def-eps-3}
    \epsilon = \frac{3^d\delta\min\{1,\mathcal{C}_k\} }{2^{3+kd}N^{m+d}\max\{T,d\}\max\{\ckunit{f},1\}}.
\end{equation}
Combining this observation with all previous steps of the proof then leads to the error bound
\begin{equation}
    \ckunit{f-\widehat{f}^N} \leq 2^k 3^d \left(1+\delta\right)  \ln^k\left(\beta N^{d+m+2}\right)\frac{\mathcal{C}_k}{N^{m-k}}, 
\end{equation}
where $k=2$ and where we define 
\begin{equation}
    \beta = \frac{5\cdot 2^{kd}\max\{T,d\}\max\{\ckunit{f},1\}}{3^d\delta\min\{1,\mathcal{C}_k\} }, 
\end{equation}
where we used that $2^5/e^2\leq 5$.

\textbf{Step 3: Estimating the network and weights sizes. } The first hidden layer requires $3\left\lceil\frac{s+n-2}{2}\right\rceil\abs{P_{s-1,d+1}}$ neurons for the computation of all multivariate monomials (cf. Lemma \ref{lem:pol-tanh}). For the computation of all $\rho^{N}_j(x_i)$ another $\sum_{i=1}^d(b_i-a_i)(N-1)$ neurons are needed in the first hidden layer. The second hidden layer needs at most $3\left\lceil\frac{d+n}{2}\right\rceil\abs{P_{d+1,d+1}}$ neurons for realizing $\widehat{\times}_{d+1}^{h}$, which needs to be performed $N^d\prod_{i=1}^d (b_i-a_i)$ times. 

In the proof we achieved the wanted accuracy by making suitable choices of $\eta, \epsilon, h$.
From equation \eqref{eqn:def-eps-3} and Lemma \ref{lem:alpha-growth}, it follows that $\alpha = O\left(N \ln(N)\right)$.

For the approximate multiplication, \eqref{eqn:def-h} requires that $h^{-1}=O(N^{d+m+2})$. Corollary \ref{lem:mult-shallow} then proves that the weights of $\widehat{\times}^h_{d+1}$ grow as $O(N^{d(d+m+2)/n})$. Finally, the condition $\eta^{-1} = O(N^{m})$ from \eqref{eqn:def-eps-2} corresponds to weights growing as $O\left(N^{m^2/n}\right)$
as a consequence of Corollary \ref{lem:pol-tanh}. This concludes the proof.
\end{proof}

\section{Bounds on the derivative of a neural network}

\begin{lemma}\label{lem:nn-cn-bound}
Let $d,n,L,W\in\mathbb{N}$ and let $u_\theta:\mathbb{R}^{d+1}\to\mathbb{R}^{d+1}$ be a neural network with $\theta\in\Theta_{L,W,R}$ for $L\geq 2$, $R,W\geq 1$, cf. Definition \ref{def:nn}. Assume that $\norm{\sigma}_{C^n}\geq 1$. Then it holds for $1\leq j \leq d+1$ that
\begin{equation}
    \norm{(u_\theta)_j}_{C^n} \leq 16^L(d+1)^{2n} \left(e^2 n^4 W^3R^n \norm{\sigma}_{C^n}\right)^{nL}\end{equation}
\end{lemma}
\begin{proof}
Using the notation of Definition \ref{def:nn}, we define the functions $F_k=\mathbb{R}^{l_{k-1}}\to\mathbb{R}$ for every $1\leq k\leq L$ as, 
\begin{equation}
    F_k = f_{L}^\theta\circ f_{L-1}^\theta \circ \cdots \circ f_k^\theta, 
\end{equation}
and note that $F_1 = u_\theta$ and $F_L = f_{L}^\theta$. 
An application of \cite[Lemma A.7]{deryck2021approximation} then brings us that
\begin{equation}
    \norm{F_k}_{C^n} \leq 16 (e^2 n^4 l_k l_{k-1}^2)^n\max_{1\leq i \leq l_k}\norm{(f^\theta_k)_i}_{C^n}^n \norm{F_{k+1}}_{C^n}. 
\end{equation}
For $R\geq 1$ and $1\leq k < L$ we find that $\norm{(f^\theta_k)_i}_{C^n} \leq R^n \norm{\sigma}_{C^n}$ for every $i$ and for $k=L$ we find that $\norm{(f^\theta_L)_i}_{C^n} \leq R (W\norm{\sigma}_{C^0}+1)$. Combining these inequalities recursively gives us
\begin{align}
    \begin{split}
         \norm{F_1}_{C^n} &\leq \norm{F_L}_{C^n} \prod_{k=1}^{L-1}\left[ 16 (e^2 n^4 l_k l_{k-1}^2)^n\max_{1\leq i \leq l_k}\norm{(f^\theta_k)_i}_{C^n}^n \right]\\
         &\leq R (W\norm{\sigma}_{C^0}+1) \left[ 16 (e^2 n^4 W^3R^n \norm{\sigma}_{C^n})^n\right]^{L-1}(d+1)^{2n}\\
         &\leq 16^L(d+1)^{2n} \left(e^2 n^4 W^3R^n \norm{\sigma}_{C^n}\right)^{nL}.
    \end{split}
\end{align}
This concludes the proof of the lemma as $F_1 = u_\theta$. 
\end{proof}

\end{document}